\documentclass[11pt,oneside]{article}
\usepackage{amsmath,amssymb,amsthm,amsfonts,amscd,soul,cite, amsxtra,amstext,mathrsfs}
\usepackage{color,enumitem,graphicx}
\usepackage[utf8]{inputenc}
\usepackage[english]{babel}
\usepackage[colorlinks=true,urlcolor=blue,
citecolor=red,linkcolor=blue,linktocpage,pdfpagelabels,
bookmarksnumbered,bookmarksopen]{hyperref}
\usepackage[english]{babel}
\usepackage{lineno}
\usepackage[hyperpageref]{backref}
\def\dis{\displaystyle}

\def\ato0{{\buildrel{\dis\longrightarrow}\over{a\to0}}}

\newtheorem{theorem}{Theorem}[section]

\newtheorem{lemma}[theorem]{Lemma}

\newtheorem{remark}{Remark}

\numberwithin{equation}{section}

\DeclareMathSymbol{\C}{\mathalpha}{AMSb}{"43}

\newtheorem{theoreml}{Theorem}

\newtheorem*{theorem-a}{Theorem A}
\newtheorem*{theorem-b}{Theorem B}

\newcommand{\bsub}{\begin{subequations}}
\newcommand{\esub}{\end{subequations}$\!$}

\begin{document}
\title{Equivalence of critical and subcritical sharp Trudinger-Moser  inequalities in fractional dimensions and extremal functions}
\author{Jos\'{e} Francisco de Oliveira
\\{\small Department of Mathematics}
\\{\small Federal University of Piau\'{\i}}\\
{\small 64049-550 Teresina, PI, Brazil}\\
{\small jfoliveira@ufpi.edu.br}
\and Jo\~{a}o Marcos do \'{O}\thanks{Second author was supported by CNPq grant 305726/2017-0}
\\{\small Department of Mathematics}
\\{\small Federal University of Para\'{\i}ba}\\
 {\small  58051-900 Jo\~{a}o Pessoa, PB, Brazil}\\
{\small jmbo@pq.cnpq.br}
}
\date{}
\maketitle
\begin{abstract}
We establish critical and subcritical sharp Trudinger-Moser inequalities for fractional dimensions on the whole space. Moreover, we obtain asymptotic lower and upper bounds for the fractional subcritical Trudinger-Moser supremum from which we can prove the equivalence between critical and subcritical inequalities.  Using this equivalence, we prove the existence of maximizers for both the subcritical and critical associated extremal problems.   As a by-product of this development, we can explicitly calculate the value of the critical supremum in some special situations.
\end{abstract}


\section{Introduction}
Let  $0<R\le\infty$, $\alpha,\;\theta\ge0$ and $q\ge 1$  are real numbers. Set $L^q_{\theta}=L^q_{\theta}(0,R)$ the weighted Lebesgue space defined as the set of all measurable functions $u$ on $(0,R)$  such that
\begin{equation}\nonumber
 \|u\|_{L^q_{\theta}}=\left\{ \begin{array}{lll}
\left(\int_0^R |u(r)|^q \, \mathrm{d}\lambda_{\theta}
\right)^{{1}/{q}}<\infty & \mbox{if} & 1\le q<\infty, \\
\mbox{ess}\;\displaystyle\sup_{0<r<R}|u(r)|<\infty &\mbox{if}&
q=\infty
\end{array}\right.
\end{equation}
where we are denoting
\begin{equation}\label{fractional integral}
\begin{aligned}
\int_{0}^{R}f(r)\mathrm{d}\lambda_{\theta}=\omega_{\theta}\int_{0}^{R} f(r)r^{\theta}\mathrm{d}r,\quad 0<R\le \infty
\end{aligned}
\end{equation} 
with  $\omega_{\theta}$ defined by
$$
\omega_{\theta}=\frac{2\pi^{\frac{\theta+1}{2}}}{\Gamma(\frac{\theta+1}{2})},\;\;\;\mbox{with}\;\;\;\Gamma(x)=\int_0^{\infty} t^{x-1} e^{-t} \,
\mathrm{d} t.
$$

In the case that $\theta$ is a positive integer number $\omega_{\theta}$ agrees precisely with the known spherical volume element for Euclidean spaces $\mathbb{R}^{\theta+1}$. In fact,  according to the formalism  in \cite{Still77},  the integration of a radially symmetric function $f (r)$  in a $(\theta+1)$-dimensional fractional space is given by \eqref{fractional integral}, when $R=\infty$.  Integration over non-integer dimensional spaces is often used in the dimensional regularization method as a powerful tool to obtain results in statistical mechanics and quantum field theory \cite{Collins, Palmer, Zubair11}.  For a deeper discussion on this subject, we suggest \cite{Zubair12} and the references therein.

We emphasize that  the Lebesgue spaces $L^{q}_{\theta}$ is also related with the classical Hardy's inequality \cite{Hardy}, see \cite{Opic, OLDOUB} for more details.  In addition,  we can use $L^{q}_{\theta}$-spaces to define Sobolev type spaces that are suitable to investigate a general class of differential operators which includes the $p$-Laplace, $ p\ge 2$ and $k$-Hessian operators in the radial form, see for instance \cite{Clement-deFigueiredo-Mitidieri, Opic,DOMADE} and references therein. Indeed,  as observed by  P. Cl\'{e}ment et al. \cite{Clement-deFigueiredo-Mitidieri}, if we  consider $X_R=X^{1,p}_R\left(\alpha,\theta\right)$,  $\alpha,\;\theta\ge0$,  $p>1$ and $0<R\le\infty$,  as the set  of all locally absolutely continuous functions on  the interval $(0,R)$ such that   $\lim_{r\rightarrow R}u(r)=0$, $u\in L^{p}_{\theta}$ and $u^{\prime}\in L^{p}_{\alpha}$, then $X_R$ becomes a Banach space endowed with the norm
\begin{equation}\label{Xnorm-full}
\|u\|=(\|u\|^{p}_{L^{p}_{\theta}}+\|u^{\prime}\|^{p}_{L^{p}_{\alpha}})^{\frac{1}{p}}.
\end{equation}
Further,  we can distinguish  two special behaviors for the weighted Sobolev spaces $X_R$. Namely, 
\textit{the Sobolev case} when the condition
\begin{equation}\label{moldura-2s}
   \framebox{$ \alpha-p+1>0$}
\end{equation}
holds and  the \textit{Trudinger-Moser case} if
\begin{equation}\label{moldura-1}
   \framebox{$\alpha-p+1=0 .$}
\end{equation}

\noindent  In the Sobolev case  \eqref{moldura-2s} the value
$$
p^{\ast}:=p^{\ast}(\alpha,p,\nu)=\frac{(\nu+1)p}{\alpha-p+1}
$$ is the  \textit{critical exponent}  for the embedding  $$X^{1,p}_R(\alpha,\theta) \hookrightarrow L^q_{\nu}.$$
Indeed,  for the bounded situation $0<R<\infty$, one has the following continuous embedding 
\begin{equation}\label{Ebeddings}
    X^{1,p}_R(\alpha,\theta) \hookrightarrow L^q_{\nu},\;\; \mbox{if }\;\; q \in \left.\left(1, p^{\ast}\right.\right]\;\; \mbox{and}\;\; \min\left\{\theta,\nu\right\}\ge\alpha-p.
\end{equation}
Moreover, in the strict case $q < p^{\ast}$, the embedding is also compact.  In contrast, for  the 
\textit{Trudinger-Moser case} one has 
the compact embedding
\begin{equation}\label{EmbeddingsTM}
X^{1,p}_R(\alpha,\theta) \hookrightarrow L^q_{\nu},\;\; \mbox{if }\;\; q\in (1,\infty) \;\; \mbox{ and }\;\; \nu\ge0.
\end{equation}
However  $X_R \hookrightarrow L^\infty_{\nu}$ does not hold, as one can see taking  $u(r)=\ln(\ln(eR/r))$.

It is worth pointing out that the weighted Sobolev spaces  $X_R$ is employed by several authors to investigate existence of solutions for a large class of differential equations.  We recommend \cite{Clement-deFigueiredo-Mitidieri, Jacobsen, deFigueiredo-Goncalves-Miyagaki,OLDOUB, Esteban1,OLTOP} for a  general class of radial operators, and for $k$-Hessian equation  \cite{RUDOL, OLDOUBJDE, OLUBrmi} and recently \cite{OLUBna}. This paper deals with intrinsic properties of  $X_R$, which are related with sharp variational inequalities.  In this direction, let us first recall some previous results. Firstly,  the embedding in \eqref{EmbeddingsTM} does not find its threshold in the weighted Lebesgue spaces $ L^q_{\nu}$, instead,   in \cite{JJ2012} it was proved a sharp  inequality of the Trudinger-Moser type (see \cite{Moser, Trudinger67}) for $X_R$ which gets embedded into an weighted Orlicz space determined by  exponential growth. In fact, let us denote 
\begin{equation}\label{volume}
\mu_{\alpha,\theta}=(\theta+1)\omega_{\alpha}^{1/\alpha}\;\;\;\mbox{and}\;\;\;
|B_{R}|_{\theta}=\int_{0}^{R}\,\mathrm{d}\lambda_{\theta}.
\end{equation}
 Then, in \cite{JJ2012} the authors proved the following:
\begin{theoreml}\label{TheoremA}
Assume $0<R<\infty$,  $\alpha\ge 1$, $\theta\ge 0$ and $p=\alpha+1$ be real numbers. Then,
\begin{description}
\item [$\mathrm{(i)}$] We have $\exp(\mu|u|^{p/(p-1)})\in L^1_{\theta}$, for any $\mu>0$ and $u\in X^{1,p}_R(\alpha,\theta)$.\\
\item [$\mathrm{(ii)}$] There exists $c>0$ depending only on $\alpha,p$ and $\theta$ such that
\begin{equation}\label{TM1}
\sup_{\|u^{\prime}\|_{L^p_{\alpha}}\le1}\frac{1}{|B_R|_{\theta}}\int_{0}^{R}
e^{\mu|u|^{\frac{p}{p-1}}}\, \mathrm{d}\lambda_{\theta}
 \quad
\left\{
  \begin{array}{lll}
    \le c & \mbox{if} & \mu\le\mu_{\alpha,\theta}\\
    =\infty & \mbox{if} & \mu>\mu_{\alpha,\theta}
  \end{array}
\right..
\end{equation}\\
\item [$\mathrm{(iii)}$] The supremum in \eqref{TM1} is attained for all $0<\mu\le \mu_{\alpha,\theta}$.
\end{description}
\end{theoreml}
In this paper we are mainly interested in the unbounded  case when $R=\infty$.  Here, according to \cite{OLTOP},  for the \textit{Sobolev case}, we also have the following continuous embedding  
\begin{equation}\label{Ebeddingswhole}
    X^{1,p}_{\infty}(\alpha,\theta) \hookrightarrow L^q_{\theta} \quad \mbox{if}\quad q \in \left.\left[p, p^{\ast}\right.\right]\quad\mbox{and}\quad\theta\ge\alpha-p.
\end{equation}
Also, the embeddings  \eqref{Ebeddingswhole} are compact under the strict conditions  $\theta>\alpha-p$ and $p<q < p^{\ast}$. In the \textit{Trudinger-Moser case} it  holds
the continuous embeddings
\begin{equation}
\label{TM-compactembeddings}
 X^{1,p}_{\infty}(\alpha,\theta) \hookrightarrow L^q_{\theta}\quad\mbox{for all}\quad q\in [p,\infty)
\end{equation}
which are compact in the   strict case $q>p$.

We recall the following  Trudinger-Moser type inequality of the scaling invariant form obtained in \cite{JJ2012}.

\begin{theoreml}\label{AT-PAMS} Assume $p\ge 2$, $\alpha=p-1$  and $\theta\ge 0$. For any $
\mu<\mu_{\alpha,\theta}$, there exists a positive constant $C_{p,\mu,\theta}$ such that, for all $u\in X^{1,p}_{\infty}(\alpha,\theta)$, $\|u^{\prime}\|_{L^p_{\alpha}}\le 1$
\begin{equation}
\int_{0}^{\infty}\varphi_{p}\left(\mu |u|^{\frac{p}{p-1}}\right)\mathrm{d}\lambda_{\theta}\le C_{p,\mu,\theta}\|u\|^{p}_{L^{p}_{\theta}},
\end{equation}
where 
\begin{equation}\label{exp-frac}
\varphi_p(t)=e^{t}-
\sum_{k=0}^{k_0-1}\frac{t^{k}}{k!}=\sum_{j\in\mathbb{N}\;:\; j\ge p-1}\frac{t^j}{j!},\; t\ge 0,
\end{equation}
with $k_0=\min\left\{j\in\mathbb{N}\;:\; j\ge p-1\right\}$.
The constant $\mu_{\alpha,\theta}$ is sharp in the sense that the supremum is infinity when $\mu\ge \mu_{\alpha,\theta}$.
\end{theoreml}
Theorem~\ref{AT-PAMS} is the fractional dimensions counterpart of the result in  S.Adachi and K. Tanaka \cite{ADTA2000}. We also refer to \cite{CAO92,DoTM,Panda} concerning the related work for the classical Sobolev spaces. Our first result in this paper yields a precise asymptotics result on the above inequality. 
\begin{theorem}\label{thm-AT}
Assume $p\ge 2$, $\alpha=p-1$  and $\theta\ge 0$. For any $
0\le\mu<\mu_{\alpha,\theta}$, we denote
\begin{equation}\nonumber
TMSC(\mu,\alpha,\theta)=\sup_{\|u^{\prime}\|_{L^{p}_{\alpha}}\le 1}\frac{1}{\|u\|^{p}_{L^{p}_{\theta}}}\int_{0}^{\infty}\varphi_{p}\left(\mu |u|^{\frac{p}{p-1}}\right)\mathrm{d}\lambda_{\theta}.
\end{equation}
Then there exist positive constants $c(\alpha,
\theta)$ and $C(\alpha,\theta)$ such that, when $\mu$ is close enough to $\mu_{\alpha,\theta}$ 
\begin{equation}\nonumber
\frac{c(\alpha,
\theta)}{1-\left(\frac{\mu}{
\mu_{\alpha,
\theta}}\right)^{p-1}}\le TMSC(\mu,\alpha,\theta)\le \frac{C(\alpha,
\theta)}{1-\left(\frac{\mu}{
\mu_{\alpha,
\theta}}\right)^{p-1}}.
\end{equation}
Moreover, the constant $\mu_{\alpha,\theta}$ is sharp in the sense that $TMSC(\mu_{\alpha,
\theta},\alpha,\theta)=\infty$.
\end{theorem}

\noindent One of the goals of this paper is to investigate the critical regime $\mu=\mu_{\alpha,\theta}$. In this case, we will firstly prove the following:
\begin{theorem} \label{thm-MR}
Assume $p\ge 2$, $\alpha=p-1$  and $\theta\ge 0$. For any $
0\le\sigma\le \mu_{\alpha,\theta}$, we denote
\begin{equation}\nonumber
TMC(\sigma, \alpha,\theta)=\sup_{\|u\|\le 1}\int_{0}^{\infty}\varphi_{p}\left(\sigma|u|^{\frac{p}{p-1}}\right)\mathrm{d}\lambda_{\theta}.
\end{equation}
Then $TMC(\sigma, \alpha,\theta)$ is finite. The constant $\mu_{\alpha,\theta}$ is sharp. In addition, we have the following identity
\begin{equation}\label{equ-identity}
TMC(\sigma, \alpha,\theta)=\sup_{\mu \in (0, \sigma)}\left(\frac{1-\left(\frac{\mu}{\sigma}\right)^{p-1}}{\left(\frac{\mu}{\sigma}\right)^{p-1}}\right)TMSC(\mu,\alpha,\theta),\;\;\mbox{for all}\;\; \sigma\le \mu_{\alpha,\theta}.
\end{equation}
\end{theorem}

For  the classical  Sobolev spaces, the critical supremum  $TMC(\sigma, \alpha,\theta)$ was first investigated by B.~Ruf in  \cite{BRuf} and Y. Li and  B.~Ruf \cite{BRufLi}. There has been a growing interest in this kind of inequalities during the last decades, and a wide literature is available, see for instance \cite{Cassani,LamLuIbero, LamLuANS, Ishi,LamLu,Ishi2} and the references therein.  We note  that the boundedness of  $TMC(\sigma, \alpha,\theta)$  has already been investigated in \cite{AbreFern}.  In this work we give a new proof for the boundedness  which enables  in particular to get a useful relation between  $TMSC(\sigma, \alpha,\theta)$ and $TMC(\sigma, \alpha,\theta)$ given by \eqref{equ-identity}.

Another interesting question about the  supremum $TMSC(\mu,\alpha,\theta)$ and  $TMC(\sigma, \alpha,\theta)$, and for Trudinger-Moser inequalities in general,  is whether extremal functions exist or not. Inspired by recent approaches in  \cite{Cassani,LamLuIbero, LamLuAdv,LamLuANS}, we will employ the identity \eqref{equ-identity}  to investigate this question. Firstly, on the subcritical supremum $TMSC(\mu,\alpha,\theta)$ we are able to prove the following:
\begin{theorem}\label{thm-AT-Maximizers}
Assume that $\alpha, p$  and $\theta$ satisfy the assumption of Theorem~\ref{thm-AT}. Then the fractional subcritical   supremum $TMSC(\mu,\alpha,\theta)$ is attained.
\end{theorem}
By using  Theorem~\ref{thm-AT-Maximizers} and the identity \eqref{equ-identity}, we will first prove the following attainability result for the fractional critical supremum  $TMC(\sigma, \alpha,\theta)$.

\begin{theorem}\label{thm-MR-Maximizers}
Assume $\alpha, p$  and $\theta$ under the assumptions of Theorem~\ref{thm-MR}. 
\begin{description}
\item [$\mathrm{(i)}$] If $k_0>p-1$  and $0<\sigma<\mu_{\alpha,\theta}$  then $TMC(\sigma, \alpha,\theta)$ is attained.
\item [$\mathrm{(ii)}$] If $k_0=p-1$  and $0<\sigma<\mu_{\alpha,\theta}$ then $TMC(\sigma, \alpha,\theta)$ is attained, whenever
$
TMC(\sigma, \alpha,\theta)>\frac{\sigma^{p-1}}{(p-1)!}.
$
\end{description}
\end{theorem}
Theorem~\ref{thm-MR-Maximizers} has already been obtained in \cite{AbreFern}, however our proof here is new and relies on the critical and subcritical equivalence given in Theorem~\ref{thm-MR}.  In addition, 
following \cite{Ishi2} we also are able to characterize precisely the attainability of $TMC(\sigma, \alpha,\theta)$ for the case $\mathrm{(ii)}$ above. In order to get this, we define the value $\sigma_{*}=\sigma_{*}(\alpha,\theta)\in [0, \mu_{\alpha,\theta})$ by
\begin{equation}\nonumber
\sigma_{*}=\inf\left\{\sigma\in (0,\mu_{\alpha,\theta})\;:\; TMC(\sigma, \alpha,\theta)\;\;\mbox{is attained}\right\}
\end{equation}
when $TMC(\sigma, \alpha,\theta)$ is attained for some $\sigma\in (0,\mu_{\alpha,\theta})$.  If $TMC(\sigma, \alpha,\theta)$ is not attained for any $\sigma\in (0,\mu_{\alpha,\theta})$ then we set $\sigma_{*}=\infty$.
\begin{theorem}\label{thm-MR-Maximizers-analyzed}
Assume that $k_0=p-1$  and $\alpha, \theta$ are as in Theorem~\ref{thm-MR}.  Suppose $\sigma_{*}<\mu_{\alpha,\theta}$. Then
\begin{description}
\item [$\mathrm{(i)}$] $ TMC(\sigma, \alpha,\theta)$ is  attained for $\sigma_{*}<\sigma<\mu_{\alpha,\theta}$.
\item [$\mathrm{(ii)}$] The function $\nu:(\sigma_{*},\mu_{\alpha,\theta})\rightarrow\mathbb{R}$ given by
$
\nu(\sigma)=\frac{(p-1)!}{\sigma^{p-1}}TMC(\sigma, \alpha,\theta)
$
 is strictly increasing. Moreover, by setting $TMC(0, \alpha,\theta)=0$, there holds
\begin{equation}\label{MTvalue}
TMC(\sigma, \alpha,\theta)\;\;\left\{\begin{aligned}
&=\frac{\sigma^{p-1}}{(p-1)!},\;\;&\mbox{for}&\;\; \sigma\in [0,\sigma_{*}] \\
&> \frac{\sigma^{p-1}}{(p-1)!},\;\;&\mbox{for }&\;\;  \sigma\in (\sigma_{*},\mu_{\alpha,\theta})
\end{aligned}\right.
\end{equation}
and in particular
\begin{equation}\label{inf-carach}
\sigma_{*}=\inf\left\{\sigma\in (0,\mu_{\alpha,\theta})\;:\; TMC(\sigma, \alpha,\theta)>\frac{\sigma^{p-1}}{(p-1)!} \right\}.
\end{equation}
\item [$\mathrm{(iii)}$] If $p>2$ we have $\sigma^*=0$ and thus $TMC(\sigma, \alpha,\theta)$ is attained for any $(0,\mu_{\alpha,\theta})$.
\end{description}
\end{theorem}
As a consequence of Theorem~\ref{thm-MR-Maximizers-analyzed}, since  $TMC(\sigma, 1,\theta)$ is not attained for $\sigma$ small enough (cf. \cite[Theorem~1.3]{AbreFern}), Theorem~\ref{thm-MR-Maximizers-analyzed} provides 
\begin{equation}
TMC(\sigma, 1,\theta)=\sup_{\|u\|\le 1}\int_{0}^{\infty}\varphi_{2}\left(\sigma|u|^{2}\right)\mathrm{d}\lambda_{\theta}=\sigma,\quad\forall\, \sigma\in [0,\sigma_{*}].
\end{equation}

The rest of this paper is arranged as follows. In Section~\ref{sec2}, we  show Theorem~\ref{thm-AT}. Section~\ref{sec3} is devoted to the subcritical and critical equivalence stated  in Theorem~\ref{thm-MR}. In Section~\ref{sec4} we will prove the  existence of extremal functions for both subcritical $TMSC$ and critical $TMC$ fractional Trudinger-Moser  supremum in  Theorem~\ref{thm-AT-Maximizers} and Theorem~\ref{thm-MR-Maximizers}. The proof of Theorem~\ref{thm-MR-Maximizers-analyzed} is given in Section~\ref{sec5}. 
\section{Sharp subcritical Trudinger-Moser inequality: Proof of Theorem~\ref{thm-AT}}
\label{sec2}
In this section, we will prove  the asymptotic behavior for the  supremum $TMSC(\mu,\alpha,\theta)$ for the subcritical Trudinger-Moser inequality in Theorem~\ref{thm-AT}.
\subsection{Some elementary properties} 

Note that from the  definition  \eqref{fractional integral}  and  the change of variables $s=\tau r$, we have
\begin{equation}\label{rts-fractional}
\begin{aligned}
\int_{0}^{\infty}f(\tau r)\mathrm{d}\lambda_{\theta}=\frac{1}{\tau^{\theta+1}}\int_{0}^{\infty} f(s)\mathrm{d}\lambda_{\theta},\quad \tau>0.
\end{aligned}
\end{equation}
Thus,   by setting  $u_{\tau}(r)=\zeta u(\tau r),$ with $\zeta, \tau>0$  and $u\in X^{1,p}_{\infty}(\alpha,\theta)$ we can write
\begin{equation}\label{rts-LpLq}
\begin{aligned}
&\|u^{\prime}_{\tau}\|^{p}_{L^{p}_{\alpha}}=\frac{(\zeta\tau)^{p}}{\tau^{\alpha+1}}\|u^{\prime}\|^{p}_{L^{p}_{\alpha}}\\
&\|u_{\tau}\|^{q}_{L^{q}_{\theta}}=\frac{\zeta^{q}}{\tau^{\theta+1}}\|u\|^{q}_{L^{q}_{\theta}} \;\; q\ge p.
\end{aligned}
\end{equation}
Also, we observe that  
\begin{equation}\label{rts-exp}
\begin{aligned}
&\varphi_p(\rho t)\le \rho^{p-1}\varphi_p(t), \;\;\mbox{if}\;\;   0\le \rho\le 1\\
&\varphi_p(\rho t)\ge \rho^{p-1}\varphi_p(t), \;\;\mbox{if}\;\;   \rho\ge 1
\end{aligned}
\end{equation}
where  $\varphi_p(t)$ is given by \eqref{exp-frac}.
\begin{lemma} \label{aux-lemma-ineq} For all $q\ge 1$ and $\epsilon>0$ it holds:
 \begin{equation}
 \left(x+y\right)^{q}\le (1+\epsilon)^{\frac{q-1}{q}}x^{q}+\left(1-(1+\epsilon)^{-\frac{1}{q}}\right)^{1-q}y^{q},\;\; x,y\ge 0.
 \end{equation}
\end{lemma}
\begin{proof}
Since $x\mapsto x^{q},\; x\ge 0$ is a convex function, we have 
\begin{equation}\nonumber
\begin{aligned}
\left(x+y\right)^{q}&=\left(\frac{1}{(1+\epsilon)^{\frac{1}{q}}}(1+\epsilon)^{\frac{1}{q}}x +\left(1-\frac{1}{(1+\epsilon)^{\frac{1}{q}}}\right)\left(1-\frac{1}{(1+\epsilon)^{\frac{1}{q}}}\right)^{-1}y\right)^{q}\\
&\le \frac{1}{(1+\epsilon)^{\frac{1}{q}}} (1+\epsilon)x^{q}+\left(1-\frac{1}{(1+\epsilon)^{\frac{1}{q}}}\right)^{1-q}y^{q}.
\end{aligned}
\end{equation}
\end{proof}

Henceforth suppose that the condition  $\alpha-p+1=0$ holds.
The next result ensures that the subcritical supremum $TMSC(\mu,\alpha,\theta)$ can be normalized.
\begin{lemma}\label{Lemma-norm1}
\begin{equation}\nonumber
TMSC(\mu,\alpha,\theta)=\sup_{\|u^{\prime}\|_{L^{p}_{\alpha}}=\|u\|_{L^{p}_{\theta}}=1}\int_{0}^{\infty}\varphi_{p}\left(\mu |u|^{\frac{p}{p-1}}\right)\mathrm{d}\lambda_{\theta}.
\end{equation}
\end{lemma}
\begin{proof}
It is sufficient to show that 
\begin{equation}\nonumber
TMSC(\mu,\alpha,\theta)\le \sup_{\|u^{\prime}\|_{L^{p}_{\alpha}}=\|u\|_{L^{p}_{\theta}}=1}\int_{0}^{\infty}\varphi_{p}\left(\mu |u|^{\frac{p}{p-1}}\right)\mathrm{d}\lambda_{\theta}.
\end{equation}
In order to get this, for each  $u\in X^{1,p}_{\infty}\setminus\left\{0\right\}$, with  $\|u^{\prime}\|_{L^{p}_{\alpha}}\le 1$ we set
$$
v(r)=\frac{u(\tau r)}{\|u^{\prime}\|_{L^{p}_{\alpha}}}; \;\;\;\mbox{with}\;\; \tau =\left(\frac{\|u\|^{p}_{L^{p}_{\theta}}}{\|u^{\prime}\|^{p}_{L^{p}_{\alpha}}}\right)^{\frac{1}{\theta+1}}.
$$
Since we are supposing $\alpha-p+1=0$,  \eqref{rts-LpLq} yields 
\begin{equation}\nonumber
\|v^{\prime}\|_{L^{p}_{\alpha}}=\|v\|_{L^{p}_{\theta}}=1.
\end{equation}
 Then, from \eqref{rts-fractional} and \eqref{rts-exp} it follows that 
\begin{equation}\nonumber
\begin{aligned}
\int_{0}^{\infty}\varphi_{p}\left(\mu |v|^{\frac{p}{p-1}}\right)\mathrm{d}\lambda_{\theta}
&=\frac{1}{\tau^{\theta+1}}\int_{0}^{\infty}\varphi_{p}\left(\frac{1 }{\|u^{\prime}\|^{\frac{p}{p-1}}_{L^{p}_{\alpha}}}\mu |u|^{\frac{p}{p-1}}\right)\mathrm{d}\lambda_{\theta}\\
& \ge \left(\frac{\|u^{\prime}\|^{p}_{L^{p}_{\alpha}}}{\|u\|^{p}_{L^{p}_{\theta}}}\right)\frac{1}{\|u^{\prime}\|^{p}_{L^{p}_{\alpha}}}\int_{0}^{\infty}\varphi_{p}\left(\mu |u|^{\frac{p}{p-1}}\right)\mathrm{d}\lambda_{\theta}\\
&=\frac{1}{\|u\|^{p}_{L^{p}_{\theta}}}\int_{0}^{\infty}\varphi_{p}\left(\mu |u|^{\frac{p}{p-1}}\right)\mathrm{d}\lambda_{\theta}
\end{aligned}
\end{equation}
which completes the proof.
\end{proof}
\subsection{Proof of Theorem~\ref{thm-AT}}
 Let $u\in X^{1,p}_{\infty}$, with $\|u^{\prime}\|_{L^{p}_{\alpha}}\le 1$. From the P\'{o}lya-Szeg\"{o} inequality obtained in \cite{AbreFern, Alvino2017}, we can assume that $u$ is a non-increasing function. Also, by Lemma~\ref{Lemma-norm1}  it is sufficient to analyze the case $\|u\|_{L^{p}_{\theta}}=1$. 
 
Initially, we will prove that 
\begin{equation}\label{AT-up}
TMSC(\mu,\alpha,\theta)\le \frac{C(\alpha,
\theta)}{1-\left(\frac{\mu}{
\mu_{\alpha,
\theta}}\right)^{p-1}}.
\end{equation} 
Let us denote by
\begin{equation}\nonumber
A_u=\left\{r>0\;:\; |u(r)|^{p}>1-\left({\mu}/{\mu_{\alpha,\theta}}\right)^{p-1}\right\}.
\end{equation}
We observe that for all $|t|\le 1$ it holds 
\begin{equation}\label{exp-estimate}
\varphi_{p}(\mu|t|^{\frac{p}{p-1}})=\sum_{j\in\mathbb{N}\;:\; j\ge p-1}\frac{\mu^{j}}{j!}|t|^{\frac{jp}{p-1}}\le \sum_{j\in\mathbb{N}\;:\; j\ge p-1}\frac{\mu^{j}}{j!}|t|^{p}\le |t|^{p}\sum_{j=0}^{\infty}\frac{\mu^{j}}{j!}=e^{\mu}|t|^{p}.
\end{equation}
Hence, if $A_u=\emptyset$ and consequently $u\le 1$ in $(0,\infty)$, the  inequality \eqref{exp-estimate} yields
\begin{equation}\label{case-below 1}
\begin{aligned}
\int_{0}^{\infty}\varphi_{p}\left(\mu|u|^{\frac{p}{p-1}}\right)\mathrm{d}\lambda_{\theta}
&\le e^{\mu}\int_{0}^{\infty}|u|^{p}\mathrm{d}\lambda_{\theta}\\
&\le \frac{e^{\mu_{\alpha,\theta}}}{1-\left(\frac{\mu}{\mu_{\alpha,\theta}}\right)^{p-1}}.
\end{aligned}
\end{equation}
So we can assume $A_u\not=\emptyset$. Thus, there exists $R_u>0$ such that 
$
A_u=(0,R_u),
$
because we are assuming $u$ is a non-increasing function. Analogously to \eqref{case-below 1}, we obtain
\begin{equation}\nonumber
\begin{aligned}
\int_{R_u}^{\infty}\varphi_{p}\left(\mu|u|^{\frac{p}{p-1}}\right)\mathrm{d}\lambda_{\theta}& \le\int_{\left\{u\le 1\right\}}\varphi_{p}\left(\mu|u|^{\frac{p}{p-1}}\right)\mathrm{d}\lambda_{\theta}\\
&\le e^{\mu}\int_{\left\{u\le 1\right\}}|u|^{p}\mathrm{d}\lambda_{\theta}\\
&\le \frac{e^{\mu_{\alpha,\theta}}}{1-\left(\frac{\mu}{\mu_{\alpha,\theta}}\right)^{p-1}}.
\end{aligned}
\end{equation}
Now, observe that
\begin{equation}\label{Bu-measure}
|B_{R_u}|_{\theta}=\int_{0}^{R_u}\mathrm{d}\lambda_{\theta}\le \frac{1}{1-\left(\frac{\mu}{\mu_{\alpha,\theta}}\right)^{p-1}}\int_{0}^{\infty}|u|^{p}\mathrm{d}\lambda_{\theta}\le \frac{1}{1-\left(\frac{\mu}{\mu_{\alpha,\theta}}\right)^{p-1}}.
\end{equation}
For $r\in (0, R_u)$, we set
\begin{equation}\nonumber
v(r)=u(r)-\left(1-\left(\frac{\mu}{\mu_{\alpha,\theta}}\right)^{p-1}\right)^{\frac{1}{p}}.
\end{equation}
It is clear that  $v\in X^{1,p}_{R_u}(\alpha,\theta)$ and $\|v^{\prime}\|_{L^{p}_{\alpha}(0,R_u)}\le 1$. Also, by choosing $\epsilon=({\mu_{\alpha,\theta}}/{\mu})^{p}-1$ and $q=p/(p-1)$ in Lemma~\ref{aux-lemma-ineq}, we have
\begin{equation}\nonumber
\begin{aligned}
|u|^{\frac{p}{p-1}}&\le (1+\epsilon)^{\frac{1}{p}}|v|^{\frac{p}{p-1}}+\left(1-\frac{1}{(1+\epsilon)^{\frac{p-1}{p}}}\right)^{-\frac{1}{p-1}}\left(1-\left(\frac{\mu}{\mu_{\alpha,\theta}}\right)^{p-1}\right)^{\frac{1}{p-1}}\\
&=\frac{\mu_{\alpha,\theta}}{\mu}|v|^{\frac{p}{p-1}}+1.
\end{aligned}
\end{equation}
Hence, the Trudinger-Moser type inequality \eqref{TM1} and $\eqref{Bu-measure}$ imply
\begin{equation}\nonumber
\begin{aligned}
\int_{0}^{R_u}\varphi_{p}\left(\mu|u|^{\frac{p}{p-1}}\right)\mathrm{d}\lambda_{\theta}& \le\int_{0}^{R_u}e^{\mu|u|^{\frac{p}{p-1}}}\mathrm{d}\lambda_{\theta}\\
&\le e^{\mu} \int_{0}^{R_u}e^{\mu_{\alpha,
\theta} |v|^{\frac{p}{p-1}}}\mathrm{d}\lambda_{\theta}\\
& \le c_{\alpha,\theta}e^{\mu}|B_{R_u}|_{\theta}\\
& \le \frac{c_{\alpha,\theta}\, e^{\mu_{\alpha,\theta}}}{1-\left(\frac{\mu}{\mu_{\alpha,\theta}}\right)^{p-1}}.
\end{aligned}
\end{equation}
\begin{remark}\label{remark-final}
At this point, we note that we have proved that 
$$
TMSC(\mu,\alpha,\theta)\le  \frac{C(\alpha,
\theta)}{1-\left(\frac{\mu}{
\mu_{\alpha,
\theta}}\right)^{p-1}}
$$
for any $\mu<\mu_{\alpha,\theta}$ not necessarily close to $\mu_{\alpha,\theta}$.
\end{remark}
This proves \eqref{AT-up}. Next,  we will prove the contrary inequality
\begin{equation}\label{AT-low}
TMSC(\mu,\alpha,\theta)\ge \frac{c(\alpha,
\theta)}{1-\left(\frac{\mu}{
\mu_{\alpha,
\theta}}\right)^{p-1}}.
\end{equation} 
To see this, let us consider the sequence
\begin{equation}\label{Msequences}
u_n(r)=\frac{1}{\omega^{\frac{1}{p}}_{\alpha}}\left\{\begin{aligned}
&\left(\frac{n}{\theta+1}\right)^{\frac{p-1}{p}}, & \mbox{if}&\quad 0\le r \le e^{-\frac{n}{\theta+1}},\\
&\left(\frac{\theta+1}{n}\right)^{\frac{1}{p}}\ln\frac{1}{r},&\mbox{if}&\quad e^{-\frac{n}{\theta+1}}<r<1,\\
& 0, &\mbox{if}& \quad r\ge 1.
 \end{aligned}\right.
\end{equation}
Since that $\alpha=p-1$, it follows that 
 \begin{equation}\nonumber
 \begin{aligned}
 &\|u^{\prime}_n\|^{p}_{L^{p}_{\alpha}}=1\\
& \|u_n\|^{p}_{L^{p}_{\theta}} =\frac{c}{n}\left[n^{p}e^{-n}+\int_{0}^{n}s^{p}e^{-s}\mathrm{d} s\right]
 \end{aligned}
 \end{equation}
 for some $c=c(\alpha,\theta)>0$. Thus,   since $\int_{0}^{\infty}s^{p}e^{-s}\mathrm{d} s=\Gamma(p+1)>0$,  there are $c_1=c_1(\alpha,\theta)>0$ and $n_1\in\mathbb{N}$ such that 
\begin{equation}\label{assyLp}
\|u_n\|^{p}_{L^{p}_{\theta}}\le \frac{c_1}{n}, \quad \forall\; n\ge n_1.
\end{equation}
On the other hand
\begin{equation}\nonumber
\begin{aligned}
\int_{0}^{\infty}\varphi_{p}\left(\mu |u_n|^{\frac{p}{p-1}}\right)\mathrm{d}\lambda_{\theta}&\ge \int_{0}^{e^{-\frac{n}{\theta+1}}}\varphi_{p}\left(\frac{\mu}{\mu_{\alpha,\theta}}n\right)\mathrm{d}\lambda_{\theta}=\frac{\omega_{\theta}}{\theta+1}\varphi_{p}\left(\frac{\mu}{\mu_{\alpha,\theta}}n\right)e^{-n}\\
&=\frac{\omega_{\theta}}{\theta+1} \left[e^{\left(\frac{\mu}{\mu_{\alpha,\theta}}-1\right)n}-\left(\sum_{j=0}^{k_0-1}\left(\frac{\mu}{\mu_{\alpha,\theta}}\right)^j\frac{n^j}{j!}\right)e^{-n}\right]\\
&\ge \frac{\omega_{\theta}}{\theta+1} \left[e^{\left(\frac{\mu}{\mu_{\alpha,\theta}}-1\right)n}-\left(\sum_{j=0}^{k_0-1}\frac{n^j}{j!}\right)e^{-n}\right].
\end{aligned}
\end{equation}
Thus, for all $n\ge n_1$
\begin{equation}\label{assyphi}
\begin{aligned}
TMSC(\mu,\alpha,\theta)&\ge \frac{1}{\|u_n\|^{p}_{L^{p}_{\theta}}}\int_{0}^{\infty}\varphi_{p}\left(\mu |u_n|^{\frac{p}{p-1}}\right)\mathrm{d}\lambda_{\theta}\\
&\ge c_2\left[n e^{\left(\frac{\mu}{\mu_{\alpha,\theta}}-1\right)n}-\left(\sum_{j=0}^{k_0-1}\frac{n^j}{j!}\right)ne^{-n}\right]\\
&= \frac{c_2}{1-\left(\frac{\mu}{\mu_{\alpha,\theta}}\right)^{p-1}} \left(1-\left(\frac{\mu}{\mu_{\alpha,\theta}}\right)^{p-1}\right) \left[ne^{-\left(1-\frac{\mu}{\mu_{\alpha,\theta}}\right)n}-\left(\sum_{j=0}^{k_0-1}\frac{n^j}{j!}\right)ne^{-n}\right],
\end{aligned}
\end{equation}
for some $c_2=c_2(\alpha,\theta)>0$. Now, we can choose $n_2\ge n_1$ such that 
\begin{equation}\nonumber
\left(1-\left(\frac{\mu}{\mu_{\alpha,\theta}}\right)^{p-1}\right)\left(\sum_{j=0}^{k_0-1}\frac{n^j}{j!}\right)ne^{-n}\le \frac{1}{e^{5}}, \quad \forall\; n\ge n_2\quad \mbox{and}\quad 0\le \mu<\mu_{\alpha,\theta}.
\end{equation}
Hence, for all $n\ge n_2$
\begin{equation}\nonumber
\begin{aligned}
TMSC(\mu,\alpha,\theta)&\ge \frac{c_2}{1-\left(\frac{\mu}{\mu_{\alpha,\theta}}\right)^{p-1}} \left[\left(1-\frac{\mu}{\mu_{\alpha,\theta}}\right) ne^{-\left(1-\frac{\mu}{\mu_{\alpha,\theta}}\right)n}-e^{-5}\right].
\end{aligned}
\end{equation}
Now,  if  $\alpha$ is close enough to $\mu_{\alpha,\theta}$ such that $\left(1-\frac{\mu}{\mu_{\alpha,\theta}}\right)^{-1}\ge n_2$, by  picking $n\in\mathbb{N}$ such that $$\left(1-\frac{\mu}{\mu_{\alpha,\theta}}\right)^{-1}\le n\le 4 \left(1-\frac{\mu}{\mu_{\alpha,\theta}}\right)^{-1}$$ we obtain
\begin{equation}\nonumber
\begin{aligned}
TMSC(\mu,\alpha,\theta)&\ge \frac{c_2}{1-\left(\frac{\mu}{\mu_{\alpha,\theta}}\right)^{p-1}} \left[e^{-4}-e^{-5}\right].
\end{aligned}
\end{equation}
Finally, from \eqref{assyphi}, for $\mu=\mu_{\alpha,\theta}$ we have
\begin{equation}\nonumber
\begin{aligned}
TMSC(\mu_{\alpha,\theta},\alpha,\theta)
&\ge c_2\left[n -\left(\sum_{j=0}^{k_0-1}\frac{n^j}{j!}\right)ne^{-n}\right]\rightarrow\infty,\;\;\mbox{as}\;\; n\rightarrow\infty.
\end{aligned}
\end{equation} 
\section{Equivalence of critical and subcritical Trudinger-Moser  inequalities}
\label{sec3}
The aim of this section is to prove the critical and subcritical  equivalence given in  Theorem~\ref{thm-MR}. We observe that we are not assuming that $TMC(\mu_{\alpha,\theta}, \alpha,\theta)$  is finite in our argument.

\begin{lemma}\label{AT-MT} For any $0<\sigma\le \mu_{\alpha,\theta}$ and $0<\mu<\sigma$
\begin{equation}\nonumber
TMSC(\mu,\alpha,\theta)\le \left(\frac{\left(\frac{\mu}{
\sigma}\right)^{p-1}}{1-\left(\frac{\mu}{
\sigma}\right)^{p-1}}\right)TMC(\sigma, \alpha,\theta).
\end{equation}
In particular, if $TMC(\mu_{\alpha,\theta}, \alpha,\theta)$  is finite, then $TMSC(\mu,\alpha,\theta)$ is finite.
\end{lemma}
\begin{proof}
Let $u\in X^{1,p}_{\infty}$, with  $\|u^{\prime}\|_{L^{p}_{\alpha}}=1$ and $\|u\|_{L^{p}_{\theta}}=1$. Set
\begin{equation}
u_{t}(r)=\left(\frac{\mu}{\sigma}\right)^{\frac{p-1}{p}}u(tr), \;\;\mbox{with}\;\; t=\left(\frac{\left(\frac{\mu}{\sigma}\right)^{p-1}}{1-\left(\frac{\mu}{\sigma}\right)^{p-1}}\right)^{\frac{1}{\theta+1}}.
\end{equation}
By \eqref{rts-LpLq} we get
\begin{equation}\nonumber
\begin{aligned}
\|u_{t}^{\prime}\|^{p}_{L^{p}_{\alpha}}&=\left(\frac{\mu}{\sigma}\right)^{p-1}\|u^{\prime}\|^{p}_{L^{p}_{\alpha}}= \left(\frac{\mu}{\sigma}\right)^{p-1}\\
\|u_{t}\|^{p}_{L^{p}_{\theta}} &=\left(\frac{\mu}{\sigma}\right)^{p-1}\frac{\|u\|^{p}_{L^{p}_{\theta}}}{t^{\theta+1}}=1-\left(\frac{\mu}{\sigma}\right)^{p-1}.
\end{aligned}
\end{equation}
Hence $\|u_{t}^{\prime}\|^{p}_{L^{p}_{\alpha}}+\|u_{t}\|^{p}_{L^{p}_{\theta}}=1$ and we have
\begin{equation}\nonumber
\begin{aligned}
\int_{0}^{\infty}\varphi_{p}\left(\mu |u|^{\frac{p}{p-1}}\right)\mathrm{d}\lambda_{\theta}& =t^{\theta+1}\int_{0}^{\infty}\varphi_{p}\left(\sigma |u_{t}|^{\frac{p}{p-1}}\right)\mathrm{d}\lambda_{\theta}\le \left(\frac{\left(\frac{\mu}{\sigma}\right)^{p-1}}{1-\left(\frac{\mu}{\sigma}\right)^{p-1}}\right)TMC(\sigma,\alpha,\theta).
\end{aligned}
\end{equation}
Since $u\in X^{1,p}_{\infty}$, with  $\|u^{\prime}\|_{L^{p}_{\alpha}}=1$ and $\|u\|_{L^{p}_{\theta}}=1$ is  arbitrary, in view of the Lemma~\ref{Lemma-norm1}, we conclude the proof.
\end{proof}
\subsection{Proof of Theorem~\ref{thm-MR}}
Let $u\in X^{1,p}_{\infty}$ such that $0<\|u^{\prime}\|^{p}_{L^{p}_{\alpha}}+\|u\|^{p}_{L^{p}_{\theta}}\le 1$. Assume that
\begin{equation}\nonumber
\|u^{\prime}\|_{L^{p}_{\alpha}}=\vartheta\in(0,1)\;\;\mbox{and}\;\; \|u\|^{p}_{L^{p}_{\theta}}\le 1-\vartheta^{p}.
\end{equation}
If $\frac{1}{2}<\vartheta<1$, we set
\begin{equation}\nonumber
u_{t}(r)=\frac{u(tr)}{\vartheta},\;\;\;\mbox{with}\;\; t=\left(\frac{1-\vartheta^{p}}{\vartheta^{p}}\right)^{\frac{1}{\theta+1}}>0.
\end{equation}
From \eqref{rts-LpLq}, we can write
\begin{equation}\nonumber
\begin{aligned}
\|u_t^{\prime}\|_{L^{p}_{\alpha}}&=\frac{\|u^{\prime}\|_{L^{p}_{\alpha}}}{\vartheta}=1\\
\|u_t\|^{p}_{L^{p}_{\theta}}& =\frac{1}{\vartheta^{p}}\frac{1}{t^{\theta+1}}\|u\|^{p}_{L^{p}_{\theta}}\le \frac{1-\vartheta^{p}}{\vartheta^{p}t^{\theta+1}}=1.
\end{aligned}
\end{equation}
Hence,  for any $\sigma\le \mu_{\alpha,\theta}$,  the Theorem~\ref{thm-AT} (cf. Remark~\ref{remark-final}) yields
\begin{equation}\nonumber
\begin{aligned}
\int_{0}^{\infty}\varphi_{p}\left(\sigma|u|^{\frac{p}{p-1}}\right)\mathrm{d}\lambda_{\theta}
&\le t^{\theta+1}\int_{0}^{\infty}\varphi_{p}\left(\vartheta^{\frac{p}{p-1}}\mu_{\alpha,\theta}|u_t|^{\frac{p}{p-1}}\right)\mathrm{d}\lambda_{\theta}\\
&\le t^{\theta+1}TMSC\left(\vartheta^{\frac{p}{p-1}}\mu_{\alpha,\theta},\alpha,\theta\right)\\
&\le \left(\frac{1-\vartheta^{p}}{\vartheta^{p}}\right)\frac{C(\alpha,\theta)}{1-\left(\frac{\vartheta^{\frac{p}{p-1}}\mu_{\alpha,\theta}}{\mu_{\alpha,\theta}}\right)^{p-1}}\\
&= \left(\frac{1-\vartheta^{p}}{\vartheta^{p}}\right)\frac{C(\alpha,\theta)}{1-\vartheta^{p}}\\
&= \frac{C(\alpha,\theta)}{\vartheta^{p}}\\
&\le 2^{p} C(\alpha,\theta).
\end{aligned}
\end{equation}
If $0<\vartheta\le \frac{1}{2}$, setting
\begin{equation}\nonumber
v(r)=2u(r/\vartheta)
\end{equation}
we have
\begin{equation}\nonumber
\begin{aligned}
& \|v^{\prime}\|_{L^{p}_{\alpha}}=2\|u^{\prime}\|_{L^{p}_{\alpha}}\le 1\\
& \|v\|^{p}_{L^{p}_{\theta}}=2^p\vartheta^{\theta+1}\|u\|^{p}_{L^{p}_{\theta}}\le 2^p\vartheta^{\theta+1}(1-\vartheta^{p})\le 2^p\vartheta^{\theta+1}.
\end{aligned}
\end{equation}
Consequently, the Theorem~\ref{thm-AT} provides
\begin{equation}\nonumber
\begin{aligned}
\int_{0}^{\infty}\varphi_{p}\left(\sigma|u|^{\frac{p}{p-1}}\right)\mathrm{d}\lambda_{\theta}
&\le \frac{1}{\vartheta^{\theta+1}}\int_{0}^{\infty}\varphi_{p}\left(2^{-\frac{p}{p-1}}\mu_{\alpha,\theta}|v|^{\frac{p}{p-1}}\right)\mathrm{d}\lambda_{\theta}\\
&\le 2^{p}TMSC\left(2^{-\frac{p}{p-1}}\mu_{\alpha,\theta},\alpha,\theta\right)\\
&\le C(\alpha,\theta)\left(\frac{2^{p}}{1-2^{-p}}\right).
\end{aligned}
\end{equation}
Since $u\in X^{1,p}_{\infty}$, with $\|u\|\le 1$ is  arbitrary, we obtain
$TMC(\sigma,\alpha,\theta)<\infty$, for any $\sigma\le \mu_{\alpha,\theta}$.

Next, we will show that the constant
 $\mu_{\alpha,\theta}$ is sharp. To see this, we  can use the  sequence $(u_n)$ in \eqref{Msequences} again. Indeed, we have
 \begin{equation}\nonumber
 \begin{aligned}
 &\|u^{\prime}_n\|^{p}_{L^{p}_{\alpha}}=1\\
& \|u_n\|^{p}_{L^{p}_{\theta}} =O\left(\frac{1}{n}\right), \;\;\mbox{as}\;\; n\rightarrow\infty.
 \end{aligned}
 \end{equation}
 Now,  for $\tau_{n}\in (0,1)$ such that
$$
\tau^{p}_{n}(1+\|u_n\|^{p}_{L^{p}_{\theta}})=1, \;\;\mbox{with}\;\; \tau_n=1-O\left(\frac{1}{n^{\frac{1}{p}}}\right)\rightarrow 1,\;\;\mbox{as}\;\; n\rightarrow\infty
$$
we set 
\begin{equation}\nonumber
v_n(r)=\tau_{n} u_n(r).
\end{equation}
Then
$$\|v^{\prime}_n\|^{p}_{L^{p}_{\alpha}}+\|v_n\|^{p}_{L^{p}_{\theta}}=\tau^{p}_n\|u^{\prime}_n\|^{p}_{L^{p}_{\alpha}}+\tau^{p}_n\|u_n\|^{p}_{L^{p}_{\theta}}=\tau^{p}_n+\tau^{p}_n\|u_n\|^{p}_{L^{p}_{\theta}}=1.$$
In addition, for any $\sigma>\mu_{\alpha,\theta}$
\begin{equation}\nonumber
\begin{aligned}
 \int_{0}^{\infty}\varphi_{p}\left(\sigma|v_n|^{\frac{p}{p-1}}\right)\mathrm{d}\lambda_{\theta}
&\ge \int_{0}^{e^{-\frac{n}{\theta+1}}}\left(e^{\frac{n\sigma}{\mu_{\alpha,\theta}}\tau^{\frac{p}{p-1}}_n}-\sum_{k=0}^{k_0-1}\frac{1}{k!}\left(\frac{n\sigma}{\mu_{\alpha,\theta}} \right)^{k}\tau^{\frac{kp}{p-1}}_{n}\right)\mathrm{d}\lambda_{\theta}\\
&=\frac{\omega_{\theta}}{\theta+1}\left[e^{\left(\frac{n\sigma}{\mu_{\alpha,\theta}}\right)\tau^{\frac{p}{p-1}}_n-n}-O\left(\frac{(n\tau_n)^{k_0-1}}{e^{n}}\right)\right]\rightarrow+\infty,\quad\mbox{as}\;\; n\rightarrow\infty.
\end{aligned}
\end{equation}
Now, we are going to show that
\begin{equation}\label{At=cMt}
TMC(\sigma,\alpha,\theta)=\sup_{\mu\in (0,\sigma)}\left(\frac{1-\left(\frac{\mu}{\sigma}\right)^{p-1}}{\left(\frac{\mu}{\sigma}\right)^{p-1}}\right)TMSC(\mu,\alpha,\theta).
\end{equation}
By Lemma~\ref{AT-MT}, we obtain
\begin{equation}\label{At<=cMT}
\sup_{\mu\in (0,\sigma)}\left(\frac{1-\left(\frac{\mu}{\sigma}\right)^{p-1}}{\left(\frac{\mu}{\sigma}\right)^{p-1}}\right)TMSC(\mu,\alpha,\theta)\le TMC(\sigma,\alpha,\theta).
\end{equation}
In order to obtain the reverse inequality, let $(u_n)$ be a maximizing sequence of  $TMC(\sigma,\alpha,\theta)$, that is,  
$u_n\in X^{1,p}_{\infty}$,  $0<\|u^{\prime}_n\|^{p}_{L^{p}_{\alpha}}+\|u_n\|^{p}_{L^{p}_{\theta}}\le 1$ such that 
\begin{equation}
TMC(\sigma,\alpha,\theta)=\lim_{n}\int_{0}^{\infty}\varphi_{p}\left(\sigma|u_n|^{\frac{p}{p-1}}\right)\mathrm{d}\lambda_{\theta}.
\end{equation}
We set 
\begin{equation}\nonumber
u_{\tau_n}(r)=\frac{u(\tau_n r)}{\|u_n^{\prime}\|_{L^{p}_{\alpha}}},\;\;\;\mbox{with}\;\; \tau_n=\left(\frac{1-\|u^{\prime}_n\|^p_{L^{p}_{\alpha}}}{\|u^{\prime}_n\|^p_{L^{p}_{\alpha}}}\right)^{\frac{1}{\theta+1}} >0.
\end{equation}
Then
\begin{equation}\nonumber
\begin{aligned}
\|u_{\tau_n}^{\prime}\|_{L^{p}_{\alpha}}&=1\\
\|u_{\tau_n}\|^{p}_{L^{p}_{\theta}}& =\frac{1}{\|u_n^{\prime}\|^{p}_{L^{p}_{\alpha}}}\frac{1}{\tau_{n}^{\theta+1}}\|u_n\|^{p}_{L^{p}_{\theta}}=\frac{\|u_n\|^{p}_{L^{p}_{\theta}}}{1-\|u^{\prime}_n\|^p_{L^{p}_{\alpha}}}\le 1.
\end{aligned}
\end{equation}
Consequently
\begin{equation}\nonumber
\begin{aligned}
& \int_{0}^{\infty}\varphi_{p}\left(\sigma |u_n|^{\frac{p}{p-1}}\right)\mathrm{d}\lambda_{\theta}= \tau^{\theta+1}_n\int_{0}^{\infty}\varphi_{p}\left(\sigma\|u^{\prime}_n\|^{{\frac{p}{p-1}}}_{L^{p}_{\alpha}} |u_{\tau_n}|^{\frac{p}{p-1}}\right)\mathrm{d}\lambda_{\theta}\\
&\le \tau^{\theta+1}_{n}TMSC\left(\sigma\|u^{\prime}_n\|^{{\frac{p}{p-1}}}_{L^{p}_{\alpha}},\alpha,\theta\right)\\
& =\left(\frac{1-\|u^{\prime}_n\|^p_{L^{p}_{\alpha}}}{\|u^{\prime}_n\|^p_{L^{p}_{\alpha}}}\right) TMSC\left(\sigma\|u^{\prime}_n\|^{{\frac{p}{p-1}}}_{L^{p}_{\alpha}},\alpha,\theta\right) \\
&=\left(\frac{1-\left(\frac{\sigma\|u^{\prime}_n\|^{\frac{p}{p-1}}_{L^{p}_{\alpha}}}{\sigma}\right)^{p-1}}{\left(\frac{\sigma\|u^{\prime}_n\|^{\frac{p}{p-1}}_{L^{p}_{\alpha}}}{\sigma}\right)^{p-1}}\right)TMSC\left(\sigma\|u^{\prime}_n\|^{{\frac{p}{p-1}}}_{L^{p}_{\alpha}},\alpha,\theta\right) \\
&\le \sup_{\mu\in (0,\sigma)}\left(\frac{1-\left(\frac{\mu}{\sigma}\right)^{p-1}}{\left(\frac{\mu}{\sigma}\right)^{p-1}}\right)TMSC(\mu,\alpha,\theta).
\end{aligned}
\end{equation}
Hence, we obtain
\begin{equation}\label{MT<=cAt}
TMC(\sigma, \alpha,\theta)\le \sup_{\mu\in (0,\sigma)}\left(\frac{1-\left(\frac{\mu}{\sigma}\right)^{p-1}}{\left(\frac{\mu}{\sigma}\right)^{p-1}}\right)TMSC(\mu,\alpha,\theta).
\end{equation}
Now, \eqref{At=cMt} follows from \eqref{At<=cMT} and \eqref{MT<=cAt}.
\section{Existence of extremal fuctions}
\label{sec4}
In this section we will prove the existence of extremal functions for both subcritical and critical Trudinger-Moser inequalities Theorem~\ref{thm-AT-Maximizers} and Theorem~\ref{thm-MR-Maximizers}. First of all, we present the following radial type Lemma.
\begin{lemma}\label{radial lemma}
For each $u\in X^{1,p}_{\infty}(\alpha,\theta)$, $p\ge2$, we have
the inequality  $$|u(r)|^{p}\le \frac{C}{r^{\frac{\alpha+\theta(p-1)}{p}}}\|u\|^{p},
\,\;\forall\,\; r>0 $$ where $C>0$ depends only on $\alpha$, $p$ and
$\theta$. In addition, 
\begin{equation}\nonumber
 \lim_{r\rightarrow\infty} r^{\frac{\alpha+\theta(p-1)}{p}}|u(r)|^{p}\rightarrow 0.
\end{equation}
\end{lemma}
\begin{proof} Let $u\in X^{1,p}_{\infty}(\alpha,\theta)$ be arbitrary.  For any $r>0$, we have 
\begin{equation}\nonumber
\begin{aligned}
 |u(r)|^p =-\int_{r}^{\infty}\frac{d}{ds}\left(|u(s)|^{p}\right)\mathrm{d} s\le  p\int_{r}^{\infty}|u(s)|^{p-1}|u^{\prime}(s)|\,\mathrm{d} s. 
\end{aligned}
\end{equation}
Hence
\begin{equation}\nonumber
\begin{aligned}
r^{\frac{\alpha+\theta(p-1)}{p}}|u(r)|^p \le  p\int_{r}^{\infty}|u(s)|^{p-1}s^{\frac{\theta(p-1)}{p}}|u^{\prime}(s)|s^{\frac{\alpha}{p}}\,\mathrm{d} s 
\end{aligned}
\end{equation}
and the Young's inequality yields
\begin{equation}\nonumber
\begin{aligned}
r^{\frac{\alpha+\theta(p-1)}{p}} |u(r)|^p \le C\left[\int_{r}^{\infty}|u(s)|^{p}\mathrm{d}\lambda_{\theta}+\int_{r}^{\infty}|u^{\prime}(s)|^{p}\mathrm{d}\lambda_{\alpha}\right],
\end{aligned}
\end{equation}
for some $C>0$ depending only on $\alpha$, $p$ and
$\theta$. This proves the result.
\end{proof}
\subsection{Maximizers for the subcritical Trudinger-Moser inequality}
Let $(u_n)\subset X^{1,p}_{\infty}$ be a maximizing sequence to the subcritical Trudinger-Moser supremum $TMSC(\mu,\alpha,\theta)$. From Lemma~\ref{Lemma-norm1}, we may suppose that
\begin{equation}\nonumber
\begin{aligned}
&TMSC(\mu,\alpha,\theta)=\lim_{n}\int_{0}^{\infty}\varphi_{p}\left(\mu |u_n|^{\frac{p}{p-1}}\right)\mathrm{d}\lambda_{\theta}\\
&\|u^{\prime}_n\|_{L^{p}_{\alpha}}=\|u_n\|_{L^{p}_{\theta}}=1\\
&u_n\rightharpoonup u\;\;\mbox{weakly in}\;\;X^{1,p}_{\infty}.
\end{aligned}
\end{equation}
From the compact embedding \eqref{TM-compactembeddings}, we also may assume that 
\begin{equation}\label{cp-conseque1}
\begin{aligned}
&u_n\rightarrow u\;\;\mbox{in}\;L^{q}_{\theta},\,q>p\;\;\mbox{and}\;\;u_n(r)\rightarrow u(r) \;\;\mbox{a.e in}\;\; (0,\infty).
\end{aligned}
\end{equation}
Of course, we also have 
\begin{equation}\nonumber
\|u^{\prime}\|_{L^{p}_{\alpha}}\le 1,\;\;\; \|u\|_{L^{p}_{\theta}}\le 1.
\end{equation}
At this point we observe that there exist $C=C(p,\mu)>0$ such that
\begin{equation}\label{phi-phiconj}
\varphi_{p}\left(\mu t^{\frac{p}{p-1}}\right)-\frac{\mu^{k_0}}{k_0!}t^{\frac{k_0p}{p-1}}\le C\varphi_{p}\left(\mu t^{\frac{p}{p-1}}\right)t^{\frac{p}{p-1}},\quad t\ge 0.
\end{equation}
Let $\epsilon>0$ be arbitrary. From Lemma~\ref{radial lemma}, there exists $R>0$ such that $|u_n(r)|\le \epsilon$, for all $r\ge R$. Hence, from \eqref{phi-phiconj} and Theorem~\ref{AT-PAMS} we obtain 
\begin{equation}\nonumber
\begin{aligned}
 \int_{R}^{\infty}\left[\varphi_{p}\left(\mu |u_n|^{\frac{p}{p-1}}\right)-\frac{\mu^{k_0}}{k_0!}|u_n|^{\frac{k_0p}{p-1}}\right]\mathrm{d}\lambda_{\theta}&\le C(p,\mu) \int_{R}^{\infty}\varphi_{p}\left(\mu |u_n|^{\frac{p}{p-1}}\right)|u_n|^{\frac{p}{p-1}}\mathrm{d}\lambda_{\theta}\\
&\le C(p,\mu) \epsilon^{\frac{p}{p-1}}  \int_{R}^{\infty}\varphi_{p}\left(\mu |u_n|^{\frac{p}{p-1}}\right)\mathrm{d}\lambda_{\theta}\\
& \le C(p,\mu,\theta) \epsilon^{\frac{p}{p-1}}.
\end{aligned}
\end{equation}
Also, we have  (cf.\eqref{cp-conseque1})
\begin{equation}\nonumber
\varphi_{p}\left(\mu |u_n|^{\frac{p}{p-1}}\right)-\frac{\mu^{k_0}}{k_0!}|u_n|^{\frac{k_0p}{p-1}}\rightarrow \varphi_{p}\left(\mu |u|^{\frac{p}{p-1}}\right)-\frac{\mu^{k_0}}{k_0!}|u|^{\frac{k_0p}{p-1}}\;\;\mbox{a.e in }\;\; (0,R),\;\; \mbox{as}\quad n\rightarrow\infty.
\end{equation}
In addition,  by setting $v_n(r)=u_n(r)-u_n(R)$ for all $r\in (0,R)$, we have $v_n\in X^{1,p}_{R}(\alpha,\theta)$ with $\|v^{\prime}_n\|_{L^{p}_{\alpha}}\le 1$. Moreover, from Lemma~\ref{aux-lemma-ineq}, for any $q>1$
\begin{equation}\nonumber
\begin{aligned}
|u_n|^{\frac{p}{p-1}}
&\le q^{\frac{1}{p}}|v_n|^{\frac{p}{p-1}}+\left(1-q^{-\frac{p-1}{p}}\right)^{-\frac{1}{p-1}}\epsilon^{\frac{p}{p-1}}.
\end{aligned}
\end{equation}
By choosing $q>1$ close to $1$ such that $q^{(p+1)/p}\mu<\mu_{\alpha,\theta}$, Theorem~\ref{TheoremA}  yields
\begin{equation}\label{bound-ball}
\begin{aligned}
 \int_{0}^{R}\left[\varphi_{p}\left(\mu |u_n|^{\frac{p}{p-1}}\right)-\frac{\mu^{k_0}}{k_0!}|u_n|^{\frac{k_0p}{p-1}}\right]^{q}\mathrm{d}\lambda_{\theta} &\le \int_{0}^{R}\left[\varphi_{p}\left(\mu |u_n|^{\frac{p}{p-1}}\right)\right]^{q}\mathrm{d}\lambda_{\theta} \\
 & \le \int_{0}^{R}e^{q\mu |u_n|^{\frac{p}{p-1}}}\mathrm{d}\lambda_{\theta}\\
 & \le C(p,q,\alpha,\theta) \int_{0}^{R}e^{\mu_{\alpha,\theta} |v_n|^{\frac{p}{p-1}}}\mathrm{d}\lambda_{\theta} \\
& \le C(p,q,\mu,\theta, R).
\end{aligned}
\end{equation}
Thus, we may use  Vitali's convergence theorem to obtain
\begin{equation}\nonumber
\begin{aligned}
 \lim_{n\rightarrow\infty}\int_{0}^{R}\left[\varphi_{p}\left[\mu |u_n|^{\frac{p}{p-1}}\right)-\frac{\mu^{k_0}}{k_0!}|u_n|^{\frac{k_0p}{p-1}}\right]\mathrm{d}\lambda_{\theta}=\int_{0}^{R}\left[\varphi_{p}\left(\mu |u|^{\frac{p}{p-1}}\right)-\frac{\mu^{k_0}}{k_0!}|u|^{\frac{k_0p}{p-1}}\right]\mathrm{d}\lambda_{\theta}.
\end{aligned}
\end{equation}
Now, using  the Brezis-Lieb lemma together with \eqref{cp-conseque1} we have
\begin{equation}\nonumber
\lim_{n\rightarrow\infty}\int_{0}^{\infty}|u_n|^{\frac{k_0p}{p-1}}\mathrm{d}\lambda_{\theta}=\left\{\begin{aligned}
 & \int_{0}^{\infty}|u|^{\frac{k_0p}{p-1}}\mathrm{d}\lambda_{\theta},\quad&\mbox{if}&\quad k_0>p-1\\
 & 1,\quad&\mbox{if}&\quad k_0=p-1.\\
\end{aligned}\right.
\end{equation}
Hence, if $k_0>p-1$
\begin{equation}\nonumber
\begin{aligned}
TMSC(\mu,\alpha,\theta)&=\lim_{n}\int_{0}^{\infty}\varphi_{p}\left(\mu |u_n|^{\frac{p}{p-1}}\right)\mathrm{d}\lambda_{\theta}\\
&=\lim_{n}\left[\int_{0}^{\infty}\left(\varphi_{p}\left(\mu |u_n|^{\frac{p}{p-1}}\right)-\frac{\mu^{k_0}}{k_0!}|u_n|^{\frac{k_0p}{p-1}}\right)\mathrm{d}\lambda_{\theta}+\frac{\mu^{k_0}}{k_0!}\int_{0}^{\infty}|u_n|^{\frac{k_0p}{p-1}}\mathrm{d}\lambda_{\theta}\right]\\
&\le \int_{0}^{R}\left(\varphi_{p}\left(\mu |u|^{\frac{p}{p-1}}\right)-\frac{\mu^{k_0}}{k_0!}|u|^{\frac{k_0p}{p-1}}\right)\mathrm{d}\lambda_{\theta}+C(p,\mu,\theta) \epsilon^{\frac{p}{p-1}}+\frac{\mu^{k_0}}{k_0!}\int_{0}^{\infty}|u|^{\frac{k_0p}{p-1}}\mathrm{d}\lambda_{\theta}\\
& \le \int_{0}^{\infty}\varphi_{p}\left(\mu |u|^{\frac{p}{p-1}}\right)\mathrm{d}\lambda_{\theta}+C(p,\mu,\theta) \epsilon^{\frac{p}{p-1}}.
\end{aligned}
\end{equation}
Setting $\epsilon\rightarrow 0$, we have
\begin{equation}\nonumber
\begin{aligned}
TMSC(\mu,\alpha,\theta)
& \le\int_{0}^{\infty}\varphi_{p}\left(\mu |u|^{\frac{p}{p-1}}\right)\mathrm{d}\lambda_{\theta}.
\end{aligned}
\end{equation}
It follows that $0<\|u\|_{L^p_{\theta}}\le 1$ and thus
\begin{equation}\nonumber
\begin{aligned}
TMSC(\mu,\alpha,\theta)
&\le \frac{1}{\|u\|^{p}_{L^{p}_{\theta}}}\int_{0}^{\infty}\varphi_{p}\left(\mu |u|^{\frac{p}{p-1}}\right)\mathrm{d}\lambda_{\theta}
\end{aligned}
\end{equation}
which completes the proof in the case $k_0>p-1$. If $k_0=p-1$, we can write
\begin{equation}\nonumber
\begin{aligned}
TMSC(\mu,\alpha,\theta)&=\lim_{n}\left[\int_{0}^{\infty}\left(\varphi_{p}\left(\mu |u_n|^{\frac{p}{p-1}}\right)-\frac{\mu^{k_0}}{k_0!}|u_n|^{p}\right)\mathrm{d}\lambda_{\theta}+\frac{\mu^{k_0}}{k_0!}\right]\\
&\le \int_{0}^{R}\left(\varphi_{p}\left(\mu |u|^{\frac{p}{p-1}}\right)-\frac{\mu^{k_0}}{k_0!}|u|^{p}\right)\mathrm{d}\lambda_{\theta}+C(p,\mu,\theta) \epsilon^{\frac{p}{p-1}}+\frac{\mu^{k_0}}{k_0!}\\
& \le  \int_{0}^{\infty}\left(\varphi_{p}\left(\mu |u|^{\frac{p}{p-1}}\right)-\frac{\mu^{k_0}}{k_0!}|u|^{p}\right)\mathrm{d}\lambda_{\theta}+C(p,\mu,\theta) \epsilon^{\frac{p}{p-1}}+\frac{\mu^{k_0}}{k_0!}.
\end{aligned}
\end{equation}
Letting $\epsilon\rightarrow0$, it follows that 
\begin{equation}\label{ATphi-term}
\begin{aligned}
TMSC(\mu,\alpha,\theta)
& \le  \int_{0}^{\infty}\left(\varphi_{p}\left(\mu |u|^{\frac{p}{p-1}}\right)-\frac{\mu^{k_0}}{k_0!}|u|^{p}\right)\mathrm{d}\lambda_{\theta}+\frac{\mu^{k_0}}{k_0!}.
\end{aligned}
\end{equation}
Moreover, for any $w\in X^{1,p}_{\infty}(\alpha,\theta)$ with $\|w^{\prime}\|_{L^{p}_{\alpha}}=\|w\|_{L^{p}_{\theta}}=1$ we have 
\begin{equation}\nonumber
\int_{0}^{\infty}\varphi_{p}\left(\mu |w|^{\frac{p}{p-1}}\right)\mathrm{d}\lambda_{\theta}\ge \frac{\mu^{k_0}}{k_0!}\int_{0}^{\infty}|w|^{p}\mathrm{d}\lambda_{\theta}+\frac{\mu^{k_0+1}}{(k_0+1)!}\int_{0}^{\infty}|w|^{\frac{p(k_0+1)}{p-1}}\mathrm{d}\lambda_{\theta}.
\end{equation}
This implies that  $TMSC(\mu,\alpha,\theta)>\frac{\mu^{k_0}}{k_0!}$. Thus, from \eqref{ATphi-term}, we get $0<\|u\|_{L^{p}_{\theta}}\le 1$
and
\begin{equation}\nonumber
\begin{aligned}
TMSC(\mu,\alpha,\theta)
& \le \frac{1}{\|u\|^{p}_{L^{p}_{\theta}}} \int_{0}^{\infty}\left(\varphi_{p}\left(\mu |u|^{\frac{p}{p-1}}\right)-\frac{\mu^{k_0}}{k_0!}|u|^{p}\right)\mathrm{d}\lambda_{\theta}+\frac{\mu^{k_0}}{k_0!}\\
&=\frac{1}{\|u\|^{p}_{L^{p}_{\theta}}}\int_{0}^{\infty}\varphi_{p}\left(\mu |u|^{\frac{p}{p-1}}\right)\mathrm{d}\lambda_{\theta}
\end{aligned}
\end{equation} 
and the result is proved.
\subsection{Maximizers for the critical Trudinger-Moser  inequality}
Next we combine the equivalence in the Theorem~\ref{thm-MR} and the Theorem~\ref{thm-AT-Maximizers} to demonstrate Theorem~\ref{thm-MR-Maximizers}. Firstly, for $0<s<\mu_{\alpha,\theta}$, we set
\begin{equation}\nonumber
f(s)= TMSC(s,\alpha,\theta)\quad\mbox{and}\quad g(s)=TMC(s, \alpha,\theta).
\end{equation}
Hence, Theorem~\ref{thm-MR} yields
\begin{equation}\label{gf-relation}
g(\sigma)=\sup_{s \in (0,\sigma)}\left(\frac{1-\left(\frac{s}{\sigma}\right)^{p-1}}{\left(\frac{s}{\sigma}\right)^{p-1}}\right)f(s).
\end{equation}
\begin{lemma}\label{f-continuous} $f$ is a continuous function on $(0,\mu_{\alpha,\theta})$.
\end{lemma}
\begin{proof}
By using Theorem~\ref{thm-AT-Maximizers}  we can pick $\epsilon_n \downarrow0$ and $u_n\in X^{1,p}_{\infty}$, with  $\|u_{n}^{\prime}\|_{L^{p}_{\alpha}}\le 1$ and $\|u_n\|_{L^{p}_{\theta}}=1$ such that 
\begin{equation}\nonumber
f(s+\epsilon_{n})=\int_{0}^{\infty}\varphi_{p}\left((s+\epsilon_n)|u_n|^{\frac{p}{p-1}}\right)\mathrm{d}\lambda_{\theta}.
\end{equation}
Then
\begin{equation}\label{fleft}
\begin{aligned}
0&\le  f(s+\epsilon_{n})- f(s)\\
&\le \int_{0}^{\infty}\left[\varphi_{p}\left((s+\epsilon_n)|u_n|^{\frac{p}{p-1}}\right)-\varphi_{p}\left(s|u_n|^{\frac{p}{p-1}}\right)\right]\mathrm{d}\lambda_{\theta}.
\end{aligned}
\end{equation}
Without loos of generality,  we also may assume that (cf. \eqref{TM-compactembeddings})
\begin{equation}\label{cp-Critical-conseq}
\begin{aligned}
&u_n\rightharpoonup u\;\;\mbox{weakly in}\;\;X^{1,p}_{\infty}\\
&u_n\rightarrow u\;\;\mbox{in}\;L^{q}_{\theta},\,q>p\;\;\mbox{and}\;\;u_n(r)\rightarrow u(r) \;\;\mbox{a.e in}\;\; (0,\infty).
\end{aligned}
\end{equation}
In particular, 
\begin{equation}\nonumber
\begin{aligned}
\varphi_{p}\left((s+\epsilon_n)|u_n(r)|^{\frac{p}{p-1}}\right)-\varphi_{p}\left(s|u_n(r)|^{\frac{p}{p-1}}\right)\rightarrow 0 \;\;\mbox{a.e in}\;\; (0,\infty).
\end{aligned}
\end{equation}
In the same way of \eqref{bound-ball}, we can use Lemma~\ref{radial lemma} and Theorem~\ref{TheoremA}  to obtain a positive constant $C(p,q,s,\theta, R)$ such that
\begin{equation}\nonumber
\begin{aligned}
\int_{0}^{R}\left[\varphi_{p}\left((s+\epsilon_n)|u_n|^{\frac{p}{p-1}}\right)-\varphi_{p}\left(s|u_n|^{\frac{p}{p-1}}\right)\right]^q\mathrm{d}\lambda_{\theta}\le C(p,q,s,\theta, R),
\end{aligned}
\end{equation}
for some  $q>1$ and for all $R>0$. It follows that
\begin{equation}\nonumber
\begin{aligned}
\int_{0}^{R}\left[\varphi_{p}\left((s+\epsilon_n)|u_n|^{\frac{p}{p-1}}\right)-\varphi_{p}\left(s|u_n|^{\frac{p}{p-1}}\right)\right]\mathrm{d}\lambda_{\theta}\rightarrow 0.
\end{aligned}
\end{equation}
On the other hand, for $R$ large enough,  Lemma~\ref{radial lemma} yields
\begin{equation}\nonumber
|u_n(r)|\le 1,\;\; \mbox{for every}\;\; n\in\mathbb{N},\;\; r\ge R.
\end{equation}
Then
\begin{equation}\nonumber
\begin{aligned}
&\int_{R}^{\infty}\left[\varphi_{p}\left((s+\epsilon_n)|u_n|^{\frac{p}{p-1}}\right)-\varphi_{p}\left(s|u_n|^{\frac{p}{p-1}}\right)\right]\mathrm{d}\lambda_{\theta}\\
&=\int_{R}^{\infty} \sum_{j\in\mathbb{N}\;:\; j\ge p-1}\left[\frac{(s+\epsilon_n)^{j}}{j!}-\frac{s^{j}}{j!}\right]|u_n|^{\frac{jp}{p-1}}\mathrm{d}\lambda_{\theta}\\
&\le \sum_{j\in\mathbb{N}\;:\; j\ge p-1}\left[\frac{(s+\epsilon_n)^{j}}{j!}-\frac{s^{j}}{j!}\right] \int_{R}^{\infty} |u_n|^{p}\mathrm{d}\lambda_{\theta}\\
&\le \left[\varphi_{p}\left(s+\epsilon_n\right)-\varphi_{p}\left(s\right)\right]\rightarrow 0.
\end{aligned}
\end{equation}
From \eqref{fleft}, we obtain  
$$
0\le  f(s+\epsilon_{n})- f(s)\rightarrow 0, \quad\mbox{as}\;\; n\rightarrow\infty.
$$
Similarly, we can also have that 
$$
0\le  f(s)- f(s-\epsilon_n)\rightarrow 0, \quad\mbox{as}\;\; n\rightarrow\infty.
$$
\end{proof}
\\
Now, in order to ensure the existence of an extremal function for  $TMC(\sigma, \alpha,\theta)$ when $0<\sigma<\mu_{\alpha,\theta}$ it is sufficient to show that 
\begin{equation}\label{closeleft}
\begin{aligned}
\limsup_{s\rightarrow 0^+}\left(\frac{1-\left(\frac{s}{\sigma}\right)^{p-1}}{\left(\frac{s}{\sigma}\right)^{p-1}}\right)f(s)<g(\sigma)
\end{aligned}
\end{equation}
and 
\begin{equation}\label{closeright}
\begin{aligned}
\limsup_{s\rightarrow \sigma^{-}}\left(\frac{1-\left(\frac{s}{\sigma}\right)^{p-1}}{\left(\frac{s}{\sigma}\right)^{p-1}}\right)f(s)<g(\sigma).
\end{aligned}
\end{equation}
Indeed, \eqref{closeleft}, \eqref{closeright} together with \eqref{gf-relation} and Lemma~\ref{f-continuous} ensure the existence of $s_{\sigma}\in (0,\sigma)$ such that  
\begin{equation}\label{gf-attained}
\begin{aligned}
g(\sigma)=\left(\frac{1-\left(\frac{s}{\sigma}\right)^{p-1}}{\left(\frac{s_{\sigma}}{\sigma}\right)^{p-1}}\right)f(s_{\sigma}).
\end{aligned}
\end{equation}
Let $u_{\sigma}$ be an extremal function for $TMSC(s_{\sigma},\alpha,\theta)$ ensured by Theorem~\ref{thm-AT-Maximizers}.  Set
$$
v_{\sigma}(r)=\left(\frac{s_{\sigma}}{\sigma}\right)^{\frac{p-1}{p}}u_{\sigma}(\tau r)
$$
where
$$
\tau=\left(\frac{\left(\frac{s_{\sigma}}{\sigma}\right)^{p-1}\|u_{\sigma}\|^{p}_{L^{p}_{\theta}}}{1-\left(\frac{s_{\sigma}}{\sigma}\right)^{p-1}}\right)^{\frac{1}{\theta+1}}.
$$
From \eqref{rts-LpLq}, it follows that 
\begin{equation}\nonumber
\begin{aligned}
\|v_{\sigma}\|^{p}=\|v^{\prime}_{\sigma}\|^{p}_{L^{p}_{\alpha}}+\|v_{\sigma}\|^{p}_{L^{p}_{\theta}}&=\left(\frac{s_{\sigma}}{\sigma}\right)^{p-1}\left[\|u^{\prime}_{\sigma}\|^{p}_{L^{p}_{\alpha}}+\tau^{-(\theta+1)}\|u_{\sigma}\|^{p}_{L^{p}_{\theta}}\right]\le 1
\end{aligned}
\end{equation}
and also we have (cf.\eqref{gf-attained})
\begin{equation}\nonumber
\begin{aligned}
TMC(\sigma, \alpha,\theta)
&=\left(\frac{1-\left(\frac{s}{\sigma}\right)^{p-1}}{\left(\frac{s_{\sigma}}{\sigma}\right)^{p-1}}\right)\frac{\tau^{\theta+1}}{\|u_{\sigma}\|^{p}_{L^{p}_{\theta}}}\int_{0}^{\infty}\varphi_{p}\left(\sigma|v_{\sigma}|^{\frac{p}{p-1}}\right)\mathrm{d}\lambda_{\theta}\\
& =\int_{0}^{\infty}\varphi_{p}\left(\sigma|v_{\sigma}|^{\frac{p}{p-1}}\right)\mathrm{d}\lambda_{\theta}.
\end{aligned}
\end{equation}
Hence, $v_{\sigma}$ is an extremal function  of $TMC(\sigma, \alpha,\theta)$.
Now, since
\begin{equation}\nonumber
\begin{aligned}
\limsup_{s\rightarrow \sigma^{-}}\left(\frac{1-\left(\frac{s}{\sigma}\right)^{p-1}}{\left(\frac{s}{\sigma}\right)^{p-1}}\right)f(s)=0<g(\sigma)
\end{aligned}
\end{equation}
it is clear that \eqref{closeright} holds.

Next, we will prove that  \eqref{closeleft} holds. Firstly, we provide the following useful estimate.
\begin{lemma}\label{TMLq} For all $q\ge p$ and $0<\mu<\mu_{\alpha,\theta}$ we have
\begin{equation}\nonumber
\sup_{\|u^{\prime}\|_{L^{p}_{\alpha}}\le 1,\;  \|u\|^{p}_{L^{p}_{\theta}}=1}\int_{0}^{\infty}\left(e^{\mu |u|^{\frac{p}{p-1}}}\right)|u|^{q}\mathrm{d}\lambda_{\theta}\le c
\end{equation}
for some $c=c(\mu,\alpha,\theta,q)>0$.
\end{lemma}
\begin{proof}
We can proceed analogous to Theorem~\ref{thm-AT}.  Indeed,  let $u\in X^{1,p}_{\infty}\setminus\left\{0\right\}$, with  $\|u^{\prime}\|_{L^{p}_{\alpha}}\le 1$ and $ \|u\|^{p}_{L^{p}_{\theta}}=1$. From the P\'{o}lya-Szeg\"{o} inequality obtained in \cite{AbreFern}, we can assume that $u$ is a non-increasing function.  We write
\begin{equation}\nonumber
\int_{0}^{\infty}\left(e^{\mu |u|^{\frac{p}{p-1}}}\right)|u|^{q}\mathrm{d}\lambda_{\theta}=\int_{\left\{u<1\right\}}\left(e^{\mu |u|^{\frac{p}{p-1}}}\right)|u|^{q}\mathrm{d}\lambda_{\theta}+\int_{\left\{u\ge 1\right\}}\left(e^{\mu |u|^{\frac{p}{p-1}}}\right)|u|^{q}\mathrm{d}\lambda_{\theta}.
\end{equation}
Of course we have 
\begin{equation}\nonumber
\int_{\left\{u<1\right\}}\left(e^{\mu |u|^{\frac{p}{p-1}}}\right)|u|^{q}\mathrm{d}\lambda_{\theta}\le e^{\mu}\int_{\left\{u<1\right\}}|u|^{q}\mathrm{d}\lambda_{\theta}\le e^{\mu}\int_{\left\{u<1\right\}}|u|^{p}\mathrm{d}\lambda_{\theta}\le e^{\mu}.
\end{equation}
Set
\begin{equation}\nonumber
I_u=\left\{r>0\;:\; u(r)\ge1\right\}.
\end{equation}
Without loss of generality, we can assume $I_u\not=\emptyset$. Thus, there is $R_u>0$ such that $I_u=(0,R_u)$. Now, if 
$$
v(r)=u(r)-1,\;\; r\in (0,R_u)
$$
we have $v\in X^{1,p}_{R_u}(\alpha,\theta)$ and $\|v^{\prime}\|_{L^{p}_{\alpha}}\le 1$.  Also,  from  Lemma~\ref{aux-lemma-ineq}  we have
\begin{equation}\nonumber
\begin{aligned}
|u|^{\frac{p}{p-1}}&\le (1+\epsilon)^{\frac{1}{p}}|v|^{\frac{p}{p-1}}+c_1(\epsilon,\alpha,\theta),
\end{aligned}
\end{equation}
for some $c_1=c_1(\epsilon,\alpha,\theta)>0$. Hence, by choosing $\epsilon>0$ small enough and $\eta>1$ such that  $\mu(1+\epsilon)^{\frac{1}{p}}\frac{\eta}{\eta-1}\le \mu_{\alpha,\theta}$, the H\"{o}lder inequality and  
 Theorem~\ref{TheoremA}  imply
\begin{equation}\label{TMLp-final}
\begin{aligned}
\int_{0}^{R_u}e^{\mu|u|^{\frac{p}{p-1}}}|u|^q\mathrm{d}\lambda_{\theta}
&\le\left(\int_{0}^{R_u}|u|^{\eta q}\mathrm{d}\lambda_{\theta}\right)^{\frac{1}{\eta}}\left(\int_{0}^{R_u}e^{\frac{\eta\mu}{\eta-1}|u|^{\frac{p}{p-1}}}\mathrm{d}\lambda_{\theta}\right)^{\frac{\eta-1}{\eta}}\\
&\le C(\epsilon,\alpha,\theta,\eta,\mu)\|u\|^{q}_{L^{\eta q}_{\theta}}\left(\int_{0}^{R_u}e^{\mu_{\alpha,\theta}|v|^{\frac{p}{p-1}}}\mathrm{d}\lambda_{\theta}\right)^{\frac{\eta-1}{\eta}}\\
&\le C(\epsilon,\alpha,\theta,\eta,\mu)\|u\|^{q}_{L^{\eta q}_{\theta}}\left(|B_{R_u}|_{\theta}\right)^{\frac{\eta-1}{\eta}}.
\end{aligned}
\end{equation}
Finally, since 
\begin{equation}\nonumber
|B_{R_{u}}|_{\theta}=\int_{0}^{R_u}\mathrm{d} \lambda_{\theta}\le \int_{0}^{R_u}|u|^{p}\lambda_{\theta}\le \|u\|_{L^{p}_{\theta}}=1
\end{equation}
and (cf. \eqref{TM-compactembeddings})  
$$
\|u\|^{q}_{L^{\eta q}_{\theta}}\le C\|u\|^q\le C(\alpha,\theta, q,\eta)
$$ 
the inequality \eqref{TMLp-final} gives the result.
\end{proof}

Since we are supposing 
$
TMC(\sigma, \alpha,\theta)>\frac{\sigma^{k_0}}{k_0!},
$
when $k_0=p-1$, to complete the proof  of \eqref{closeleft}, and then the proof  of Theorem~\ref{thm-MR-Maximizers},  it is now enough to prove the following:
\begin{lemma}\label{MTcotaS} For each $0<\sigma<\mu_{\alpha,\theta}$, we have 
\begin{equation}\nonumber
\limsup_{s\rightarrow 0^+}\left[\frac{1-\left(\frac{s}{\sigma}\right)^{p-1}}{\left(\frac{s}{\sigma}\right)^{p-1}}\right]f(s)\;\;\left\{\begin{aligned}
&=0,\;\;&\mbox{if}&\;\; k_0>p-1 \\
&\le \frac{\sigma^{k_0}}{k_0!},\;\;&\mbox{if}&\;\; k_0=p-1.
\end{aligned}\right.
\end{equation}
\end{lemma}
\begin{proof}
Let $(s_n)$ be an arbitrary sequence that $s_n\downarrow 0$. From Theorem~\ref{thm-AT-Maximizers}, we can find a sequence  $(u_n)\subset X^{1,p}_{\infty}$, with  $\|u_{n}^{\prime}\|_{L^{p}_{\alpha}}\le 1$ and $\|u_n\|_{L^{p}_{\theta}}=1$ such that
\begin{equation}\nonumber
\begin{aligned}
f(s_{n})&=\frac{s^{k_0}_n}{k_0!}\int_{0}^{\infty}|u_n|^{\frac{k_0p}{p-1}}\mathrm{d}\lambda_{\theta}+s^{k_0+1}_{n}\sum_{j\ge k_0+1}\int_{0}^{\infty}\frac{s^{j-(k_0+1)}_n}{j!}|u_n|^{\frac{jp}{p-1}}\mathrm{d}\lambda_{\theta}\\
&=\frac{s^{k_0}_n}{k_0!}\int_{0}^{\infty}|u_n|^{\frac{k_0p}{p-1}}\mathrm{d}\lambda_{\theta}+s^{k_0+1}_{n}\sum_{\ell=0}^{\infty}\int_{0}^{\infty}\frac{s^{\ell}_n}{(\ell+k_0+1)!}|u_n|^{\frac{\ell p}{p-1}+\frac{(k_0+1)p}{p-1}}\mathrm{d}\lambda_{\theta}.
\end{aligned}
\end{equation}
Since $(\ell+k_0+1)!\ge \ell!$ and in view of Lemma~\ref{TMLq}, we can write
\begin{equation}\nonumber
\begin{aligned}
f(s_{n})&\le \frac{s^{k_0}_n}{k_0!}\int_{0}^{\infty}|u_n|^{\frac{k_0p}{p-1}}\mathrm{d}\lambda_{\theta}+s^{k_0+1}_{n}\int_{0}^{\infty}\left(e^{\sigma|u_n|^{\frac{p}{p-1}}}\right)|u_n|^{\frac{(k_0+1)p}{p-1}}\mathrm{d}\lambda_{\theta}\\
&\le  \frac{s^{k_0}_n}{k_0!}\int_{0}^{\infty}|u_n|^{\frac{k_0p}{p-1}}\mathrm{d}\lambda_{\theta}+s^{k_0+1}_{n}c(\sigma,\alpha,\theta).\\
\end{aligned}
\end{equation}
It follows that 
\begin{equation}\label{f2parts}
\begin{aligned}
\left(\frac{1-\left(\frac{s_n}{\sigma}\right)^{p-1}}{\left(\frac{s_n}{\sigma}\right)^{p-1}}\right)f(s_n)&\le \left(\frac{1-\left(\frac{s_n}{\sigma}\right)^{p-1}}{\left(\frac{1}{\sigma}\right)^{p-1}}\right)\frac{s^{k_0-(p-1)}_n}{k_0!}\left[\int_{0}^{\infty}|u_n|^{\frac{k_0p}{p-1}}\mathrm{d}\lambda_{\theta}+ c(\sigma,\alpha,\theta)k_0! s_n\right].
\end{aligned}
\end{equation}

\paragraph{Case~1:} $k_0>p-1$\\

From  \eqref{TM-compactembeddings} and \eqref{f2parts}, we obtain 
\begin{equation}\nonumber
\begin{aligned}
\left(\frac{1-\left(\frac{s_n}{\sigma}\right)^{p-1}}{\left(\frac{s_n}{\sigma}\right)^{p-1}}\right)f(s_n) &\rightarrow 0,\quad \mbox{as}\;\; n\rightarrow\infty.
\end{aligned}
\end{equation}

\paragraph{Case~2:} $k_0=p-1$\\

We  have
\begin{equation}\nonumber
\begin{aligned}
\left(\frac{1-\left(\frac{s_n}{\sigma}\right)^{p-1}}{\left(\frac{s_n}{\sigma}\right)^{p-1}}\right)f(s_n)&\le \left(\frac{1-\left(\frac{s_n}{\sigma}\right)^{p-1}}{\left(\frac{1}{\sigma}\right)^{p-1}}\right)\frac{1}{k_0!}\left[1+ c(\sigma,\alpha,\theta)s_n\right]\\
&\rightarrow  \frac{\sigma^{k_0}}{k_0!},\quad \mbox{as}\;\; n\rightarrow\infty.
\end{aligned}
\end{equation}
Hence, the proof is completed.
\end{proof}
\section{Proof of Theorem~\ref{thm-MR-Maximizers-analyzed}}
\label{sec5}
In this section, we will analyze the attainability of $TMC(\sigma, \alpha,\theta)$ when the condition  $k_0=p-1$ holds.
\begin{lemma}\label{vanishing-integer} For any $\sigma\in (0, \mu_{\alpha,\theta})$ we have
\begin{equation}\label{MT>sigma}
TMC(\sigma, \alpha,\theta)\ge \frac{\sigma^{k_0}}{k_0!}.
\end{equation}
In addition, if $p>2$ the above inequality becomes strict.
\end{lemma}
\begin{proof}
We follow the argument of Ishiwata \cite{Ishi}.
Let $u\in X^{1,p}_{\infty}(\alpha,\theta)$ such that $\|u\|=1$, and  set $$u_t(r)=t^{1/p}u(t^{\frac{1}{\theta+1}}r).$$
We can easily show that
\begin{equation}\nonumber
\begin{aligned}
&\|u^{\prime}_t\|^{p}_{L^{p}_{\alpha}}=t\|u^{\prime}\|_{L^{p}_{\alpha}}^{p}\\
 &\|u_t\|^{q}_{L^{q}_{\theta}}=t^{\frac{q-p}{p}}\|u\|_{L^{q}_{\theta}}^{q},\;\; \forall\, q\ge p.
\end{aligned}
\end{equation}
In particular, for each $t>0$ small enough, if $\xi_{t}=(t+(1-t)\|u\|^{p}_{L^{p}_{\theta}})^{-1/p}$ we have 
$$
\|\xi_{t} u_t\|^{p}=t\xi^{p}_{t}\|u^{\prime}\|_{L^{p}_{\alpha}}^{p}+\xi^{p}_{t}\|u\|_{L^{p}_{\theta}}^{p}=1.
$$
Noticing that $\xi^{p}_{t}\rightarrow1/\|u\|^{p}_{L^{p}_{\theta}}$ as $t\rightarrow 0^{+}$, then for $v_t=\xi_{t}u_t$ we have
\begin{equation}\label{MT>sigmacadeia}
\begin{aligned}
TMC(\sigma, \alpha,\theta)&\ge  \int_{0}^{\infty}\varphi_{p}\left(\sigma\left|v_t\right|^{\frac{p}{p-1}}\right)\mathrm{d}\lambda_{\theta}\\
&\ge \frac{\sigma^{k_0}}{k_0!}\int_{0}^{\infty}|v_t|^{p}\mathrm{d}\lambda_{\theta}+\frac{\sigma^{k_0+1}}{(k_0+1)!}\int_{0}^{\infty}|v_t|^{\frac{p(k_0+1)}{p-1}}\mathrm{d}\lambda_{\theta}\\
&= \frac{\sigma^{p-1}}{(p-1)!}\left[\xi^{p}_{t}\|u\|^{p}_{L^{p}_{\theta}}+\frac{\sigma}{p}\xi^{\frac{p}{p-1}+p}_{t}\|u\|^{\frac{p^2}{p-1}}_{L^{\frac{p^2}{p-1}}_{\theta}}t^{\frac{1}{p-1}}\right]\\
& \rightarrow \frac{\sigma^{p-1}}{(p-1)!},\quad\mbox{as}\quad t\rightarrow 0.
\end{aligned}
\end{equation}
This proves \eqref{MT>sigma}. Moreover, if $p>2$ we  observe that the function
\begin{equation}\nonumber
h_{p,\theta,\sigma}(t)=\xi^{p}_{t}\|u\|^{p}_{L^{p}_{\theta}}+\frac{\sigma}{p}\xi^{\frac{p}{p-1}+p}_{t}\|u\|^{\frac{p^2}{p-1}}_{L^{\frac{p^2}{p-1}}_{\theta}}t^{\frac{1}{p-1}}
\end{equation}
satisfies $h_{p,\theta,\sigma}(0)=1$ and $h^{\prime}_{p,\theta,\sigma}(t)>0$ for $t>0$ small enough. Hence, the result follows from \eqref{MT>sigmacadeia}.

\end{proof}
\begin{lemma} \label{lemma-analyzed} 
\begin{description}
\item [$\mathrm{(i)}$] The function $\sigma\mapsto \frac{(p-1)!}{\sigma^{p-1}}TMC(\sigma,\alpha,\theta) $ is non-decreasing for $0<\sigma\le \mu_{\alpha,\theta}$.
\item [$\mathrm{(ii)}$]  Let $0<\sigma_1<\sigma_2\le \mu_{\alpha,\theta}$. Suppose that
$TMC(\sigma_{1},\alpha,\theta)$  is attained. Then
\begin{equation}\nonumber
\frac{(p-1)!}{\sigma^{p-1}_{2}}TMC(\sigma_{2},\alpha,\theta) >\frac{(p-1)!}{\sigma^{p-1}_{1}}TMC(\sigma_{1},\alpha,\theta)
\end{equation}
and $TMC(\sigma_{2},\alpha,\theta)$ is also attained.
\end{description}
\end{lemma}
\begin{proof}
$\mathrm{(i)}$ Since
\begin{equation}\nonumber
\frac{(p-1)!}{\sigma^{p-1}}\varphi_{p}\left(\sigma|t|^{\frac{p}{p-1}}\right)=(p-1)!\sum_{j=p-1}^{\infty}\frac{\sigma^{j-(p-1)}}{j!}t^{\frac{jp}{p-1}}
\end{equation}
it is clear that for all $t\not=0$
\begin{equation}\label{varphiMono}
\frac{(p-1)!}{\sigma^{p-1}_{1}}\varphi_{p}\left(\sigma_{1}|t|^{\frac{p}{p-1}}\right)<\frac{(p-1)!}{\sigma^{p-1}_{2}}\varphi_{p}\left(\sigma_{2}|t|^{\frac{p}{p-1}}\right),\quad 0<\sigma_1<\sigma_2\le \mu_{\alpha,\theta}.
\end{equation}
Thus, $\mathrm{(i)}$ is proved.\\
$\mathrm{(ii)}$ Since $TMC(\sigma_{1},\alpha,\theta)$  is attained, we can pick $u\in X^{1,p}_{\infty}$ such that $\|u\|=1$ and $$TMC(\sigma_{1},\alpha,\theta)=\int_{0}^{\infty}\varphi_{p}\left(\sigma_1|u|^{\frac{p}{p-1}}\right)\mathrm{d}\lambda_{\theta}.$$
Thus, the Lemma~\ref{vanishing-integer} and  \eqref{varphiMono} yield
\begin{equation}\nonumber
\begin{aligned}
\frac{(p-1)!}{\sigma^{p-1}_2}TMC(\sigma_2, \alpha,\theta)&\ge\frac{(p-1)!}{\sigma^{p-1}_2}  \int_{0}^{\infty}\varphi_{p}\left(\sigma_{2}\left|u\right|^{\frac{p}{p-1}}\right)\mathrm{d}\lambda_{\theta}\\
&> \frac{(p-1)!}{\sigma^{p-1}_1}  \int_{0}^{\infty}\varphi_{p}\left(\sigma_{1}\left|u\right|^{\frac{p}{p-1}}\right)\mathrm{d}\lambda_{\theta}\\
&=\frac{(p-1)!}{\sigma^{p-1}_1}TMC(\sigma_1, \alpha,\theta)\ge 1.
\end{aligned}
\end{equation}
Then, we have  $\frac{(p-1)!}{\sigma^{p-1}_2}TMC(\sigma_2, \alpha,\theta)>1$ and thus Theorem~\ref{thm-MR-Maximizers}-$\mathrm{(ii)}$ asserts that  $TMC(\sigma_2, \alpha,\theta)$ is attained.
\end{proof}
\subsubsection*{Proof of Theorem~\ref{thm-MR-Maximizers-analyzed} completed}
$\mathrm{(i)}$ It follows directly from Lemma~\ref{lemma-analyzed} and the definition of $\sigma_{*}$.\\

\noindent$\mathrm{(ii)}$ From  Lemma~\ref{lemma-analyzed} the function $\sigma\mapsto \frac{(p-1)!}{\sigma^{p-1}}TMC(\sigma,\alpha,\theta)$ is strictly increasing on $(\sigma_{*},\mu_{\alpha,\theta})$.  Next, we will show that
\begin{equation}\label{sigma*}
TMC(\sigma_{*},\alpha,\theta)=\frac{\sigma^{p-1}_{*}}{(p-1)!}.
\end{equation}
For our convention  $TMC(0,\alpha,\theta)=0$,  we may assume $\sigma_{*}\in (0,\mu_{\alpha,\theta})$. From Lemma~\ref{vanishing-integer}, if \eqref{sigma*} is not true we must have
$$
TMC(\sigma_{*},\alpha,\theta)>\frac{\sigma^{p-1}_{*}}{(p-1)!}.
$$ 
Thus, since $\sigma_{*}<\mu_{\alpha}$, Theorem~\ref{thm-MR-Maximizers}-$\mathrm{(ii)}$ implies that $TMC(\sigma_{*},\alpha,\theta)$ is achieved for some   $u_{*}\in X^{1,p}_{\infty}$.  Also, we have
\begin{equation}\nonumber
\lim_{\sigma\rightarrow \sigma_{*}}\int_{0}^{\infty}\varphi_{p}\left(\sigma|u_{*}|^{\frac{p}{p-1}}\right)\mathrm{d}\lambda_{\theta}=\int_{0}^{\infty}\varphi_{p}\left(\sigma_{*}|u_{*}|^{\frac{p}{p-1}}\right)\mathrm{d}\lambda_{\theta}=TMC(\sigma_{*},\alpha,\theta)>\frac{\sigma^{p-1}_{*}}{(p-1)!}.
\end{equation}
Hence, if $\sigma\in (0,\sigma_{*})$ is sufficiently close to $\sigma_{*}$, we must have
\begin{equation}\nonumber
TMC(\sigma,\alpha,\theta)\ge \int_{0}^{\infty}\varphi_{p}\left(\sigma|u_{*}|^{\frac{p}{p-1}}\right)\mathrm{d}\lambda_{\theta}>\frac{\sigma^{p-1}_{*}}{(p-1)!}>\frac{\sigma^{p-1}}{(p-1)!}.
\end{equation}
Thus, for such a $\sigma\in (0,\sigma_{*})$, Theorem~\ref{thm-MR-Maximizers}-$\mathrm{(ii)}$ implies that $TMC(\sigma,\alpha,\theta)$ is achieved which contradicts the definition of $\sigma_{*}$. This proves \eqref{sigma*}. Now,  from \eqref{sigma*} and Lemma~\ref{lemma-analyzed}-$\mathrm{(ii)}$, for each $\sigma\in(\sigma_{*}, \mu_{\alpha,\theta})$,  the supremum $TMC(\sigma,\alpha,\theta)$ is attained and  we also  have
\begin{equation}\label{MT>}
\frac{(p-1)!}{\sigma^{p-1}}TMC(\sigma,\alpha,\theta) >\frac{(p-1)!}{\sigma^{p-1}_{*}}TMC(\sigma_{*},\alpha,\theta)=1.
\end{equation}
In addition, Lemma~\ref{vanishing-integer}, Theorem~\ref{thm-MR-Maximizers}-$\mathrm{(ii)}$ and the definition of $\sigma_{*}$ yield 
\begin{equation}\label{MT=}
TMC(\sigma,\alpha,\theta)=\frac{\sigma^{p-1}}{(p-1)!},\quad\mbox{for each}\;\; \sigma\in [0,\sigma_{*}].
\end{equation}
Now, it is clear that \eqref{MT>} and \eqref{MT=} give \eqref{MTvalue}. Finally, let us denote
$$
\overline{\sigma}_{*}=\inf\left\{\sigma\in (0,\mu_{\alpha,\theta})\;:\; TMC(\sigma, \alpha,\theta)>\frac{\sigma^{p-1}}{(p-1)!} \right\}.
$$
Then, Theorem~\ref{thm-MR-Maximizers}-$\mathrm{(ii)}$ yields  $\sigma_{*}\le \overline{\sigma}_{*}$.  If $\sigma_{*}<\overline{\sigma}_{*}$ we can pick $\sigma_{0}\in (\sigma_{*}, \overline{\sigma}_{*})$ for which we must have 
\begin{equation}\nonumber
\frac{(p-1)!}{\sigma^{p-1}_{0}}TMC(\sigma_0,\alpha,\theta) >\frac{(p-1)!}{\sigma^{p-1}_{*}}TMC(\sigma_{*},\alpha,\theta)=1,
\end{equation}
that is,
\begin{equation}\nonumber
TMC(\sigma_0,\alpha,\theta) >\frac{\sigma^{p-1}_{0}}{(p-1)!}
\end{equation}
which contradicts the definition of $\overline{\sigma}_{*}$. Hence \eqref{inf-carach} holds. Finally, $\mathrm{(iii)}$ follows directly from Lemma~\ref{vanishing-integer}. 

\bigskip
\bigskip
\end{document}